\pdfoutput=1
\documentclass{amsart}
\usepackage{research}
\usepackage{subcaption}

\newcommand{\std}{\mathrm{std}}

\makeatletter
\def\renewtheorem#1{
  \expandafter\let\csname#1\endcsname\relax
  \expandafter\let\csname c@#1\endcsname\relax
  \gdef\renewtheorem@envname{#1}
  \renewtheorem@secpar
}
\def\renewtheorem@secpar{\@ifnextchar[{\renewtheorem@numberedlike}{\renewtheorem@nonumberedlike}}
\def\renewtheorem@numberedlike[#1]#2{\newtheorem{\renewtheorem@envname}[#1]{#2}}
\def\renewtheorem@nonumberedlike#1{  
\def\renewtheorem@caption{#1}
\edef\renewtheorem@nowithin{\noexpand\newtheorem{\renewtheorem@envname}{\renewtheorem@caption}}
\renewtheorem@thirdpar
}
\def\renewtheorem@thirdpar{\@ifnextchar[{\renewtheorem@within}{\renewtheorem@nowithin}}
\def\renewtheorem@within[#1]{\renewtheorem@nowithin[#1]}
\makeatother

\renewtheorem{remark}[theorem]{Remark}
\RequirePackage{tgpagella}
\RequirePackage[utf8]{inputenc} 
\RequirePackage[T1]{fontenc}
\RequirePackage{amssymb}
\RequirePackage{microtype}

\usepackage{xcolor}
\pagecolor{white}

\newcommand{\RR}{{\mathbb R}}

\newcommand{\ZZ}{{\mathbb Z}}

\newcommand{\Mxi}{(M,\xi)}

\newcommand{\Sdom}{(\Sigma,\beta)}

\newcommand{\bbr}{\begin{rem}\em} 
	\newcommand{\eer}{\end{rem}}

\begin{document}
\title[A JSJ-type decomposition theorem for symplectic fillings]{A JSJ-type decomposition theorem for symplectic fillings of contact 3-manifolds}
\author{Austin Christian and Michael Menke}
\begin{abstract}
We establish a JSJ-type decomposition theorem for splitting exact symplectic fillings of contact 3-manifolds along \emph{mixed tori} --- these are convex tori satisfying a particular geometric condition.  As an application, we show that if $(M,\xi)$ is obtained from $(S^3,\xi_{\mathrm{std}})$ via Legendrian surgery along a knot which has been stabilized both positively and negatively, then $(M,\xi)$ has a unique exact filling.
\end{abstract}
\maketitle
\vspace{-5\medskipamount}

\section{Introduction}\label{sec:intro}
A fundamental question in contact geometry is to determine the symplectic fillings of a given contact manifold.  That is, to what extent does the boundary of a symplectic manifold determine its interior?  The goal of this paper is to explain how to decompose an exact or weak symplectic filling whose boundary contact manifold contains a \emph{mixed torus} --- a convex torus admitting certain bypasses.  We then use this decomposition result to show that certain contact manifolds which are obtained as Legendrian surgeries are uniquely exactly fillable.\\

Recall that an \emph{exact symplectic filling} of a contact 3-manifold $(M,\xi)$ is a four dimensional symplectic manifold-with-boundary $(W,\omega)$ such that $\partial W = M$, $\omega = d\alpha$ for some 1-form $\alpha$ on $W$, and $\alpha|_{\partial W}$ is a positive contact form for $\xi$.  We call $(W,\omega)$ a \emph{weak symplectic filling} of $(M,\xi)$ under the more relaxed hypotheses that $\partial W = M$ and $\omega|_\xi>0$.\\

Let us start with a partial list of known results classifying the number of exact symplectic fillings of a given contact manifold. A detailed survey can be found in \cite{ozbagci2015topology}.
\begin{itemize}
\item  (Eliashberg \cite{eliashberg1990filling}) The standard contact 3-sphere $(S^3,\xi_{\std})$ has a unique exact filling up to symplectomorphism. 

\item (Wendl \cite{wendl2010strongly}) The 3-torus $(T^3,\xi_1)$, where $\xi_1$ is the canonical contact structure on the unit cotangent bundle of $T^2$, has a unique exact filling up to symplectomorphism.  Stipsicz \cite{stipsicz2002gauge} previously showed that, up to homeomorphism, $(T^3,\xi_1)$ admits a unique exact filling. 

\item (McDuff \cite{mcduff1990structure}) The standard tight contact structure on $L(p,1)$ has a unique exact filling up to symplectomorphism when $p \neq 4$; for $p = 4$, there are precisely two exact fillings.

\item (Lisca \cite{lisca2008symplectic})  Each lens space $L(p,q)$ with its canonical contact structure admits finitely many exact symplectic fillings.  Lisca gives a precise catalog of these fillings.

\item (Plamenevskaya and Van Horn-Morris \cite{plamenevskaya2010planar}, Kaloti \cite{kaloti2013stein}) Lens spaces of the form $L(p(m+1)+1,m+1)$, with virtually overtwisted contact structures, admit unique exact fillings. The case $L(p,1)$ is shown in \cite{plamenevskaya2010planar} and the general case in \cite{kaloti2013stein}.

\item (Sivek and Van Horn-Morris \cite{sivek2017fillings}) Fillings for the unit cotangent bundle of an orientable surface are unique up to $s$-cobordism, and similar results for non-orientable surfaces were proven by Li and Ozbagci \cite{li2018fillings}. 

\item (Akhmedov, Etnyre, Mark, Smith \cite{akhmedov2008note}) It is not always the case that there is a unique exact filling, or even finitely many. 
\end{itemize}

Before stating our main theorem, let us introduce \emph{mixed tori}, and say what it means to split a contact manifold along a mixed torus.  We will call a convex torus $T\subset(M,\xi)$ a \emph{mixed torus} if there exists a neighborhood $T^2\times[0,2]$ of $T$ in $M$ so that $T=T^2\times\{1\}$, the restriction of $\xi$ to $T^2\times[0,2]$ is virtually overtwisted, and each of $T^2\times[0,1]$ and $T^2\times[1,2]$ is a basic slice.  (We will recall these contact geometry notions in Section~\ref{sec:background}.)\\

Now suppose $T^2\times[0,2]\subset (M,\xi)$ is a neighborhood witnessing the existence of a mixed torus $T = T^2\times\{1\}$, and identify $T$ with $\mathbb{R}^2/\mathbb{Z}^2$ in such a way that the dividing curves of $T$ have slope $\infty$.  For $i=0,1,2$, we will denote the slope of the dividing curves of $T^2\times\{i\}$ by $s_i$ --- so, for instance, $s_1=\infty$ --- and we may normalize the identification $T=\mathbb{R}^2/\mathbb{Z}^2$ so that $s_0=-1$.  With this identification understood, we will now define the closed contact manifold $(M',\xi')$ which results from \emph{splitting $M$ along $T$ with slope $s$}, for any integer $s\in\mathbb{Z}$.  Topologically, we obtain $M'$ by attaching a solid torus to $M\setminus T$ along each of the two torus components of $\partial(M\setminus T)=T_0\sqcup T_1$:
\[
M' = S_0 \cup_{\psi_0} (M\setminus T) \cup_{\psi_1} S_1.
\]
We will define the contact structure $\xi'$ to be $\xi$ on $M\setminus T$, while on $S_i\subset M'$ we use the unique tight contact structure determined by the characteristic foliation of $\partial S_i$.  The remaining ambiguity lies in the maps $\psi_i\colon\partial S_i\to T_i$.  These are chosen so that the image of the meridian of $S_i$ is a curve of slope $s$ in $T_i=\mathbb{R}^2/\mathbb{Z}^2$.\\

Notice that if $(M',\xi')$ is the result of splitting $(M,\xi)$ along $T$ with some slope, then we may recover $(M,\xi)$ from $(M',\xi')$ by removing a pair of solid tori and identifying the dividing sets and meridians of their boundary tori.  This relationship between $(M,\xi)$ and $(M',\xi')$ allows us to obtain a filling of $(M,\xi)$ from any filling of $(M',\xi')$ via \emph{round symplectic 1-handle attachment}.  A round symplectic 1-handle is a symplectic manifold-with-boundary diffeomorphic to $S^1\times D^1\times D^2$, carrying a Liouville vector field $Z$ which is inward-pointing along $S^1\times S^0\times D^2$ and outward-pointing along $S^1\times D^1\times S^1$.  The vector field $Z$ induces the standard contact structure on the boundary solid tori $S^1\times S^0\times D^2$, and is thus attached to a symplectic filling by identifying two standard solid tori in the boundary of the filling.  For instance, if $(W',\omega')$ is a strong symplectic filling of $(M',\xi')$, we may attach a round symplectic 1-handle to $(W',\omega')$ along standard neighborhoods of the core Legendrians $L_0\subset S_0$ and $L_1\subset S_1$ to obtain $(W,\omega)$, a strong symplectic filling of $(M,\xi)$.  This construction will be explained in greater detail in Section~\ref{sec:background}.\\

We are now prepared to state our main theorem, which says that every filling of $(M,\xi)$ may be constructed as above.

\begin{theorem} \label{thm:main-thm}
	Let $(M,\xi)$ be a closed, cooriented 3-dimensional contact manifold and let $(W,\omega)$ be an exact (respectively, weak) symplectic filling of $(M,\xi)$. If there exists a mixed torus $T \subset (M,\xi)$, witnessed by a neighborhood $T^2\times[0,2]$ with slopes $s_0=-1$, $s_1=\infty$, and $s_2$, then there exists a (possibly disconnected) symplectic manifold $(W',\omega')$ such that:
	\begin{itemize}
		\item $(W',\omega')$ is an exact (respectively, weak) filling of its boundary $(M',\xi')$;
		\item $(M',\xi')$ is the result of splitting $(M,\xi)$ along $T$ with some slope $s$ satisfying $0\leq s\leq s_2-1$;
		\item $(W,\omega)$ can be recovered from $(W',\omega')$ by round symplectic 1-handle attachment.
	\end{itemize}
\end{theorem}

\begin{remark}\label{remark:main-thm-hypotheses}
The condition that $T$ be a mixed torus is essential; the theorem is not true if one only assumes that $T$ is a convex torus with two homotopically essential dividing curves.
\end{remark}

\begin{remark}
An earlier version of this paper stated Theorem~\ref{thm:main-thm} for exact and strong symplectic fillings, rather than exact and weak symplectic fillings.  The authors thank Hyunki Min for observing that the result is not true for strong fillings (c.f. \cite[Theorem 1.4]{min2022strongly}), and that our proof can be modified to hold for weak fillings.
\end{remark}

There are two results we must mention here to properly contextualize Theorem~\ref{thm:main-thm}.  The first is due to Eliashberg \cite{eliashberg1990filling}, who showed that if $(M,\xi)$ is a closed contact 3-manifold obtained from $(M',\xi')$ (which may be disconnected) via a connected sum, then every symplectic filling of $(M,\xi)$ is obtained from such a filling of $(M',\xi')$ by attaching a Weinstein 1-handle.  Our result is in the same spirit, replacing the holomorphic discs of Eliashberg's proof with holomorphic annuli, and thus considering an embedded torus in $(M,\xi)$ rather than an embedded sphere.\\

This leads us to the next result of historical importance.  Work of Jaco-Shalen \cite{jaco1978new} and Johannson \cite{johannson1979homotopy} established the JSJ decomposition for irreducible, orientable, closed 3-manifolds.  For such a manifold $M$, this decomposition identifies a (unique up to isotopy) minimal collection $T\subset M$ of disjoint, incompressible tori for which each component of $M\setminus N(T)$ is either atoroidal or Seifert-fibered, where $N(T)\subset M$ is an open tubular neighborhood of $T$.  One can think of this decomposition as a toroidal version of the prime decomposition of 3-manifolds, which allows us to write every compact, orientable 3-manifold as the connected sum of a unique collection of prime 3-manifolds.  Though the analogy is admittedly strained --- the tori in the JSJ decomposition being incompressible, for instance --- the present paper takes its name from the fact that our result replaces the connected sums of Eliashberg with toroidal decompositions.\\

In order to apply Theorem~\ref{thm:main-thm}, we will observe that mixed tori naturally appear in contact manifolds which result from Legendrian surgery along a Legendrian knot which has been stabilized both positively and negatively.  This allows us to prove the following consequence of Theorem~\ref{thm:main-thm}.

\begin{theorem} \label{thm:twice-stabilized}
Let $L$ be an oriented Legendrian knot in a closed cooriented contact 3-manifold $(M,\xi)$. Let $(M',\xi')$ be the manifold obtained from $(M,\xi)$ by Legendrian surgery on $S_+S_-(L)$, where $S_+$ and $S_-$ are positive and negative stabilizations, respectively. Then every exact (respectively, weak) filling of $(M',\xi')$ is obtained from an exact (respectively, weak) filling of $(M,\xi)$ by attaching a symplectic 2-handle along $S_+S_-(L)$. 
\end{theorem}

In particular, the following corollary holds when $(M,\xi) = (S^3,\xi_{std})$, since $(S^3,\xi_{std})$ has a unique exact filling.

\begin{corollary} \label{cor:main-cor}
If $(M',\xi')$ is obtained from $(S^3,\xi_{std})$ by Legendrian surgery on $S_+S_-(L)$, then $(M',\xi')$ has a unique exact filling up to symplectomorphism.
\end{corollary}

\begin{remark}
The assumption in Corollary~\ref{cor:main-cor} that $L$ is stabilized both positively and negatively is crucial (c.f. Remark~\ref{remark:main-thm-hypotheses}).  For example, we may obtain $L(4,1)$ with its standard contact structure via Legendrian surgery along an unknot which has been positively (or negatively) stabilized twice.  But, as mentioned above, McDuff showed in \cite{mcduff1990structure} that $(L(4,1),\xi_{\mathrm{std}})$ has two distinct exact fillings.
\end{remark}

Kaloti and Li \cite{kaloti2013surgeries} had previously shown the uniqueness up to symplectomorphism of exact fillings of manifolds obtained from Legendrian surgery along certain 2-bridge and twist knots and their stabilizations.\\

Related results were established by Lazarev for higher dimensions in \cite{lazarev2020contact}. While not stated in quite the same manner, the main result of \cite{lazarev2020contact} involves surgery on loose Legendrians. We observe that in dimensions greater than or equal to 5 all Legendrians which have been stabilized near a cusp edge are loose; their analog in dimension 3 is a Legendrian which has been stabilized both positively and negatively.

\subsection*{Acknowledgments} The authors would like to thank John Etnyre, Ko Honda, Hyunki Min, and Burak Ozbagci for many helpful discussions and suggestions.  Additionally, the authors express sincere gratitude to an anonymous referee, whose careful reading of the paper led to many improvements.  The first author was partially supported by NSF grant DMS-1745583.

\section{Background}\label{sec:background}

\subsection{Contact geometry preliminaries} \label{convex}
A knot $L \subset (M,\xi)$ is called \emph{Legendrian} if it is everywhere tangent to the contact structure $\xi$. The front projection of a Legendrian knot in $(\mathbb{R}^3, \ker(dz - y dx))$ is its projection to the $xz$-plane.  We stabilize $L$ by locally adding a zigzag in its front projection.  There are two stabilizations of $L$ --- denoted $S_+(L)$ and $S_-(L)$ --- with front projections as given in Figure~\ref{fig:Stab}.\\

\begin{figure}
	\centering
	\includegraphics[width=3in]{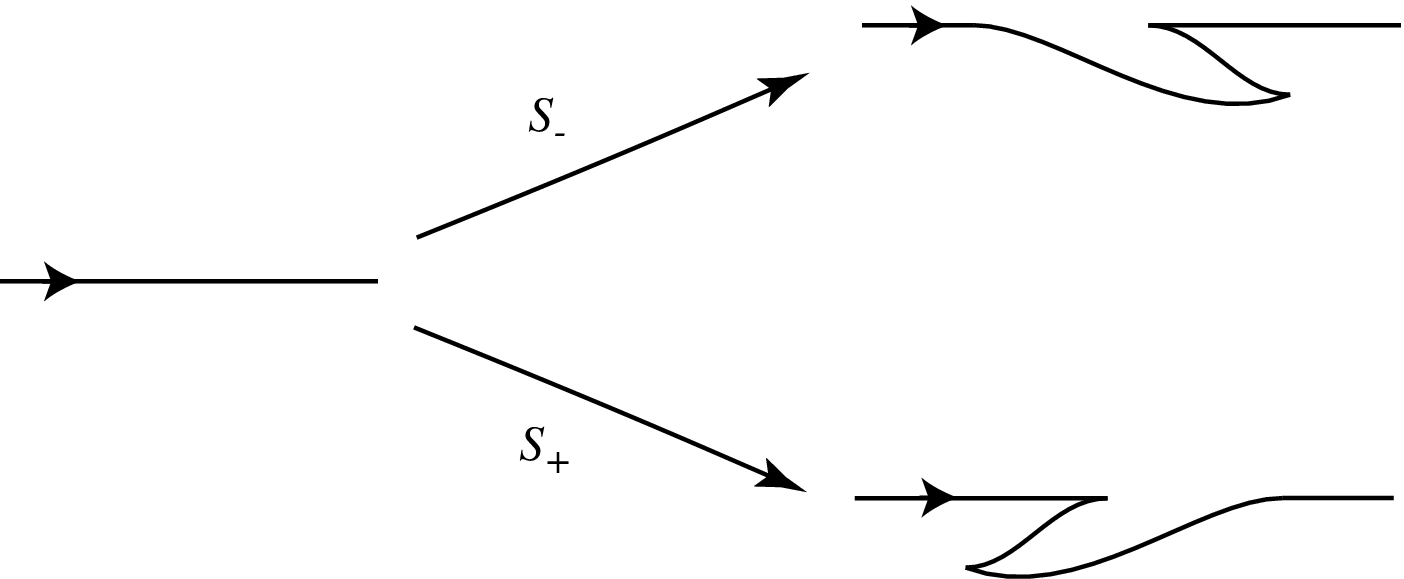}
	\caption{Stabilizations}
	\label{fig:Stab}
\end{figure}

Next, we recall notions of convexity as described by Giroux in \cite{giroux1991convexite}.  A properly embedded, oriented surface $\Sigma$ in $(M,\xi)$ is called \emph{convex} if there is a vector field $v$ transverse to $\Sigma$ whose flow preserves $\xi$.  We denote by $\Gamma_\Sigma(v)$ the set of points $x\in\Sigma$ where $v(x)$ lies in the contact plane $\xi(x)$.  The set $\Gamma_\Sigma(v)$ is a disjoint union of properly embedded smooth curves and arcs which are transverse to the \emph{characteristic foliation} $\xi|_\Sigma$.  If $\Sigma$ is closed, $\Gamma_\Sigma(v)$ will only contain closed curves, and the isotopy type of $\Gamma_\Sigma(v)$ will be independent of the choice of $v$.  For this reason, we will slightly abuse notation and refer to $\Gamma_\Sigma(v)$ as \emph{the dividing set of $\Sigma$} --- with no reference to $v$ --- and use the notation $\Gamma_\Sigma$.  In fact, when there is no ambiguity in $\Sigma$, we will write $\Gamma$ for $\Gamma_\Sigma$.\\

For any $x\in\Sigma\setminus\Gamma_\Sigma$, the vector $v(x)$ is transverse to both $T_x\Sigma$ and $\xi(x)$, and thus inherits an orientation from each of these spaces. We denote by $R_+$ and $R_-$ the regions in $\Sigma$ where these orientations agree and disagree, respectively, and write
\[
\Sigma\setminus\Gamma_\Sigma = R_+ \cup (-R_-).
\]
Finally, we recall that convexity is characterized by the existence of a neighborhood on which the contact structure is vertically invariant.  That is, every convex surface admits a standard neighborhood $\Sigma \times [-\epsilon,\epsilon] \subset (M,\xi)$ such that $\Sigma = \Sigma \times \{0\}$ and on this neighborhood $\alpha$ can be written as $\alpha = g\,dt + \beta$, where $g: \Sigma \to \mathbb{R}$ is a smooth function, $\beta$ is a 1-form on $\Sigma$, and $\Gamma = \{g = 0\}$.\\

As a final fundamental notion, recall that a \emph{standard neighborhood} $N(L)$ of a Legendrian knot $L$ is a sufficiently small tubular neighborhood of $L$ whose torus boundary is convex and whose dividing set has two components. If $S_\pm(L)$ is a stabilization of $L$, then $N(S_\pm(L))$ can be viewed as a subset of $N(L)$. Fix an oriented identification $\partial N(L) \simeq \mathbb{R}^2/\mathbb{Z}^2$ such that slope$(\Gamma_{\partial N(L)}) = \infty$ and slope$(\mathrm{meridian}) = 0$. Then slope$(\Gamma_{\partial N(S_\pm(L))}) = -1$.

\subsection{Bypasses}
In~\cite{honda2000classification}, Honda defined \emph{bypasses} to study the failure of convexity for a 1-parameter family of surfaces.  A \emph{bypass disk} $D$ for a Legendrian knot $L$ is a disk whose boundary is the union of two Legendrian arcs $a$ and $b$ such that
\begin{itemize}
	\item $a = L \cap D \subset L$.
	\item Along $a$ there are three elliptic singularities, two at the endpoints of $a$ with the same sign, and one in the middle with the opposite sign.
	\item Along $b$ there are at least 3 singularities all of the same sign.
	\item There are no other singularities in $D$.
\end{itemize}

\begin{remark}
We may define a new Legendrian knot $L'$ by smoothing $(L-a)\cup b$ and observe that $L$ is a stabilization of $L'$, with sign determined by the sign of the middle singularity on $a$.  We say that $D$ is a \emph{stabilizing disk} for $L'$.
\end{remark}
	   
The following theorem, due to Honda \cite{honda2000classification}, explains how a bypass changes the dividing set of a surface:

\begin{theorem}[{\cite[Lemma 3.12]{honda2000classification}}]\label{honda:bypass-nbhd}
Let $\Sigma$ be a convex surface, $D$ a bypass disk along $a \subset \Sigma$. Inside any open neighborhood of $\Sigma \cup D$ there is a one-sided neighborhood $\Sigma \times [0,1]$ such that $\Sigma = \Sigma \times \{0\}$ and $\Gamma_\Sigma$ is related to $\Gamma_{\Sigma \times \{1\}}$ by Figure~\ref{fig:bypass-honda}.
	
	\begin{figure}[!ht]
		\centering
		\includegraphics[width=3in]{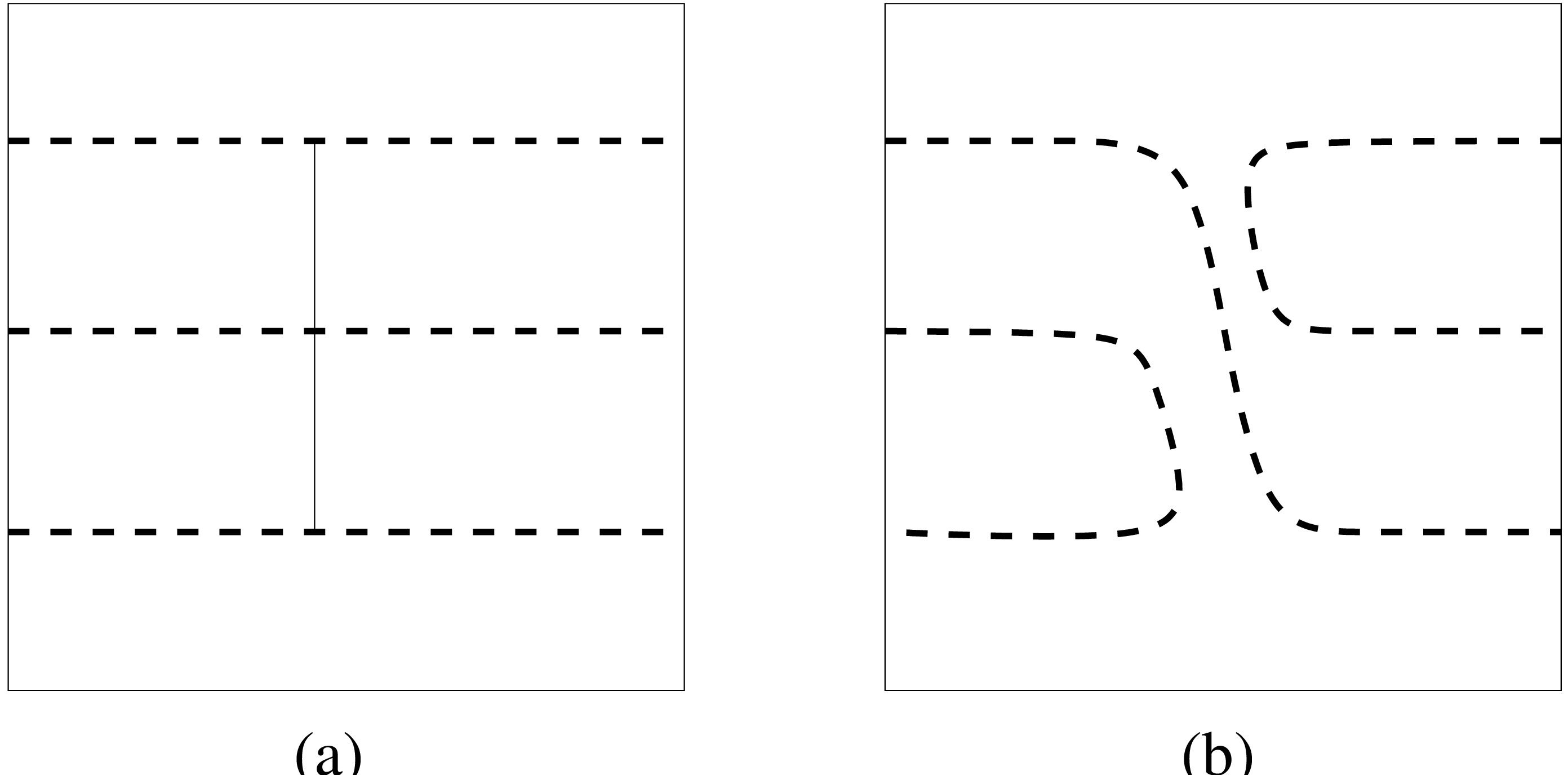}
		\caption{On the left is the dividing set of $\Sigma$ with solid attaching arc $a$. On the right is the result of bypass attachment.}
		\label{fig:bypass-honda}
	\end{figure}
\end{theorem}

When Theorem~\ref{honda:bypass-nbhd} is applied, we say that $\Sigma \times \{1\}$ is obtained from $\Sigma$ by a \emph{bypass attachment}. If the endpoints of the Legendrian arc $a$ lie on the dividing set $\Gamma$ of $\Sigma$ then we say the bypass is attached along $\Gamma$.

\subsection{Basic slices} \label{subsec:basic}
In the case where our surface is the torus $T^2 \simeq \mathbb{R}^2/\mathbb{Z}^2$, bypasses give rise to another fundamental object defined in \cite{honda2000classification}: the \emph{basic slice}.  Consider a tight $(T^2 \times I,\xi)$, where $I = [0,1]$, with convex boundary where both boundary components have two homotopically non-trivial dividing curves. Let $s_0$ and $s_1$ be the slopes of the dividing curves on $T^2 \times \{0\}$ and $T^2 \times \{1\}$, respectively, and let $v_i$ be a minimal length integral vector with slope $s_i$, for $i=0,1$.  We call $(T^2\times I,\xi)$ a \emph{basic slice} if
\begin{itemize}
	\item the vectors $v_0$ and $v_1$ form an integral basis for $\mathbb{Z}^2$;
	\item the slopes of all dividing curves on convex tori parallel to $T^2\times \{0,1\}$ have slopes on $[s_1,s_0]$ if $s_1<s_0$ and on $[s_1,\infty]\cup [-\infty,s_0]$ if $s_0< s_1$.
\end{itemize}
It was shown by Giroux~\cite{giroux2000structures} and Honda~\cite{honda2000classification} that there are exactly two tight contact structures on a given basic slice. They are distinguished by their relative Euler class.\\

Given that $T^2 \times [0,1]$ and $T^2 \times [1,2]$ are basic slices, we would like to know whether $T^2 \times [0,2]$ is universally tight. Let $s_0,s_1,s_2$, the slopes of the dividing sets on $T^2 \times \{0,1,2\}$, be $-2,-1,0$ respectively. Then $T^2 \times [0,2]$ is universally tight if the relative Euler class $e(\xi,s)$ is nonzero, where $s$ is a nowhere zero section of $\xi$ on the boundary.

\begin{definition}
A convex torus $T^2 \times \{1\} = T^2 \subset (M,\xi)$ is a \emph{mixed torus} if there exist basic slices $T^2 \times [0,1]$ and $T^2 \times [1,2]$ such that $T^2 \times [0,2]$ is not universally tight.
\end{definition}

In other words, $T^2$ is a mixed torus if the basic slices $T^2\times[0,1]$ and $T^2\times[1,2]$ are algebraically cancelling, in the sense that, with respect to a nowhere zero section of $\xi$ along $T^2\times\{0,1,2\}$, the relative Euler classes of $T^2\times[0,1]$ and $T^2\times[1,2]$ sum to zero.

\subsection{Contact handles} \label{subsec:contact-handles}
Let $D$ be a bypass disk, $\Sigma\subset(M,\xi)$ a convex surface. There is a correspondence between attaching $D$ to $\Sigma$ and attaching a pair of topologically canceling handles to a one-sided neighborhood $N(\Sigma)=\Sigma\times[0,\epsilon]$ of $\Sigma$.  In particular, the coordinates on $N(\Sigma)$ may be chosen so that $\xi$ has the form $\ker(dt+\beta)$, for some $\beta\in\Omega^1(\Sigma)$.  Then the boundary components $\Sigma\times\{0,\epsilon\}$ are convex, with dividing sets $\Gamma_\Sigma\times\{0,\epsilon\}$.  By attaching a \emph{contact 1-handle} followed by a \emph{contact 2-handle} to $N(\Sigma)$ along $\Sigma\times\{\epsilon\}$, we obtain a neighborhood of the form $\Sigma\times[0,1]$, as described in Theorem~\ref{honda:bypass-nbhd}.\\

We now give a brief description of standard models for contact handles of index 1 and 2.  While we follow closely the exposition in \cite{ozbagci2011contact}, the ideas are originally due to Giroux \cite{giroux1991convexite}.\\

Each of our standard models may be realized as a subset of $(\mathbb{R}^3,\ker\alpha)$, where $\alpha=dz+ydx+2xdy$.  Consider the subsets
\[
H_1=\{(x,y,z)\in\mathbb{R}^3|x^2+z^2\leq\epsilon,y^2\leq 1\}
\]
and
\[
H_2=\{(x,y,z)\in\mathbb{R}^3|x^2+z^2\leq 1,y^2\leq\epsilon\},
\]
for some small $\epsilon>0$.  We have a contact vector field
\[
Z=2x\frac{\partial}{\partial x}-y\frac{\partial}{\partial y}+z\frac{\partial}{\partial z}
\]
on $(\mathbb{R}^3,\ker\alpha)$ which witnesses the convexity of $\partial H_1$ and $\partial H_2$.  The respective dividing sets are
\[
\partial H_1\cap\{z=0\}
\quad\text{and}\quad
\partial H_2\cap\{z=0\}.
\]
Finally, we say that $(H_1,\ker\alpha)$ is our standard model for a contact 1-handle, with attaching region $\partial H_1\cap\{y=\pm 1\}$ --- where $Z$ is inward-pointing --- while $(H_2,\ker\alpha)$ is our standard contact 2-handle, attached using $-Z$, so that $\partial H_2\cap\{x^2+z^2=1\}$ is our attaching region.\\

As with topological 1-handle attachment, we attach a contact 1-handle to a 3-dimensional contact manifold-with-convex-boundary $(M,\xi)$ by identifying the attaching region of $(H_1,\ker\alpha)$ with regular neighborhoods of a pair of points $p,q\in\partial M$.  The difference from topological handle attachment is that we require this identification to identify the dividing sets of the regular neighborhoods with that of the attaching region.  This allows the contact structures on $(M,\xi)$ and $(H_1,\ker\alpha)$ to be identified, so that handle attachment yields a contact manifold.  Contact 2-handle attachment is analogous.

\subsection{Sutured contact manifolds}\label{subsec:sutured}
There are circumstances in which we want handle attachment to extend not only our contact structure $\xi$ on, but also a preferred contact form for $\xi$.  Gluing together contact manifolds-with-boundary in a manner which produces a compatible contact form is generally difficult, but \cite{colin2011sutures} gives us the language of \emph{sutured contact manifolds} with which we may perform this gluing.  Here we record --- and then proceed to abuse --- some of this language.\\

\begin{definition}
A \emph{sutured $3$-manifold} is a triple $(M,\Gamma,U(\Gamma))$, where
\begin{enumerate}
	\item $M$ is a compact, oriented 3-manifold-with-corners;
	\item $\Gamma\subset\partial M$ is a multicurve, called the \emph{suture};
	\item $U(\Gamma)\simeq[-1,1]\times[-1,0]\times\Gamma$ is a neighborhood of $\Gamma=\{(0,0)\}\times\Gamma$ in $M$ such that:
	\begin{enumerate}
		\item $U\cap\partial M = ([-1,1]\times\{0\}\times\Gamma)\cup(\{-1\}\times[-1,0]\times\Gamma)\cup(\{1\}\times[-1,0]\times\Gamma)$;
		\item $\partial M - ([-1,1]\times\{0\}\times\Gamma)$ is the disjoint union of a pair of surfaces $R_+(\Gamma)$ and $R_-(\Gamma)$, with the orientation of $\partial M$ agreeing with that of $R_+(\Gamma)$ and opposite that of $R_-(\Gamma)$;
		\item the orientation of $\Gamma$ agrees with the boundary orientation of $R_{\pm}(\Gamma)$;
		\item the corners of $M$ are $\{\pm 1\}\times\{0\}\times\Gamma$.
	\end{enumerate}
\end{enumerate}
\end{definition}

Note that this is the definition given in \cite[Section 2.3]{colin2011sutures}, which differs from the usual definition given by Gabai in \cite{gabai1983foliations}.\\

A prototypical sutured $3$-manifold may be constructed from a surface-with-boundary $\Sigma$ by defining $M=[-1,1]_t\times\Sigma$, $\Gamma=\{0\}\times\partial\Sigma$, and $U(\Gamma)=[-1,1]\times N(\partial\Sigma)$, for some collar neighborhood $N(\partial\Sigma)\subset\Sigma$.  Moreover, if $\beta$ is a Liouville form for $\Sigma$ which makes $N(\partial\Sigma)$ a symplectization over $\partial\Sigma$, then $\alpha=dt+\beta$ defines a contact form on $M$ such that each of $(R_+(\Gamma),\alpha|_{R_+(\Gamma)})$ and $(R_-(\Gamma),\alpha|_{R_-(\Gamma)})$ is a Liouville manifold and the Reeb vector field of $\alpha$ points from $R_-(\Gamma)$ to $R_+(\Gamma)$ along $U(\Gamma)$.  This example leads to the notion of a \emph{sutured contact manifold}.\\

\begin{definition}
Let $(M,\Gamma,U(\Gamma))$ be a sutured $3$-manifold, $\alpha$ a contact form on $M$, and $\gamma\subset\Gamma$ a component of the suture.  We say that $(M,\Gamma,U(\Gamma))$ is \emph{convex along $\gamma$ with respect to $\alpha$} if, using coordinates $(t,\tau,x)$ on the neighborhood $U(\gamma)=[-1,1]\times[-1,0]\times\gamma$, the following hold:
\begin{enumerate}
	\item $(R_+(\Gamma),\alpha|_{R_+(\Gamma)})$ and $(R_-(\Gamma),\alpha|_{R_-(\Gamma)})$ are Liouville manifolds;
	\item $\alpha=C\,dt+\beta$ inside $U(\Gamma)$, where $C>0$ is constant and $\beta$ is independent of $t$ and $\beta(\partial_t)=0$;
	\item $\partial_\tau$ is the Liouville vector field for $\beta$ in $U(\Gamma)$.
\end{enumerate}
We say that $(M,\Gamma,U(\Gamma))$ is \emph{concave along $\gamma$ with respect to $\alpha$} if the same conditions hold, except that for the second condition we instead have a coordinate system $([-2,2]_t\times[-1,1]_\tau - (-1,1)_t\times(0,1]_\tau)\times\gamma$ on $U(\gamma)$.\\
\end{definition}

The fundamental distinction between convex and concave boundary conditions on $(M,\Gamma,U(\Gamma))$ is whether the Reeb orbits of $\alpha$ along the suture run from $R_-(\Gamma)$ to $R_+(\Gamma)$ (in the convex case) or from $R_+(\Gamma)$ to $R_-(\Gamma)$ (in the concave case).  In both cases, the Reeb vector field is parallel to $\partial M$ along the suture.

\begin{remark}
In \cite{colin2011sutures}, Colin-Ghiggini-Honda-Hutchings define a sutured contact manifold\footnote{Note the absence of the adjective convex.} to be a sutured manifold for which every component of the suture is convex with respect to a fixed contact form; a \emph{concave} sutured contact manifold instead requires every component of the suture to be concave with respect to the contact form.  The definition we give here allows for contact manifolds with a mixture of concave and convex sutures, as these are natural to the argument we pursue.
\end{remark}

\subsection{Legendrian surgery}
Let $L$ be a Legendrian knot in $(M,\xi)$ with standard neighborhood $N(L)$. Topologically, Legendrian surgery is a $\mathrm{tb}(L) - 1$ Dehn surgery on $L$ and we then take care that the contact structures agree on the boundary.\\

More precisely, pick an oriented identification of $\partial N(L)$ with $\mathbb{R}^2/\mathbb{Z}^2$ so that $\pm(1,0)^T$ is the meridian and $\pm (0,1)^T$ corresponds to slope of $\Gamma_{N(L)}$. Identifying $\partial {M\setminus N(L)}$ with
$-\partial N(L)$, we can define maps
\[
\phi_\pm\colon\partial (D^2\times S^1)\to \partial ({M\setminus N(L)})
\]
on the topological level by 
\[
\phi(x,y)= \begin{pmatrix} 1 & 0\\ \pm 1 & 1\end{pmatrix}\begin{pmatrix} x \\ y \end{pmatrix}.
\] 
Let $M_\pm(L)$
be the manifold obtained by gluing $D^2\times S^1$ to $M\setminus N(L)$ using this map. The contact structure $\xi$ restricts to a contact structure $\xi|_{M\setminus N(L)}$ on $M\setminus N(L)$ and the two dividing curves on $\partial (M\setminus N(L)),$ as seen on $\partial (D^2\times S^1),$ represent $(\mp 1, 1)$ curves. Thus, according to \cite{kanda1997classification}, there is a unique tight contact structure on $D^2\times S^1$ having convex boundary with these dividing curves. Hence we may extend $\xi|_{M\setminus N(L)}$ to a contact structure 
$\xi_\pm$ on $M_\pm.$ The contact manifold $(M_\pm,\xi_\pm)$ is said to be obtained from $(M,\xi)$ by \emph{$\pm 1$-contact surgery} on $L$. The term \textit{Legendrian surgery} refers to $-1$-contact surgery. 

\subsection{Symplectization}
Let $(M,\xi)$ be a 3-dimensional contact manifold with contact form $\alpha$. The \emph{symplectization} of $(M,\xi)$ is the symplectic manifold $(\mathbb{R} \times M, d(e^s\alpha))$, where $s$ is the $\mathbb{R}$ coordinate. Given a strong symplectic filling $(W,\omega)$ of $(M,\xi)$, we can form the completion $(\widehat W,\widehat \omega)$ of $W$ by attaching $([0,\infty) \times M, d(e^s\alpha))$ to $M = \partial W$, where $\omega = d\alpha$ on $M \times \{0\}$. We will refer to $([0,\infty) \times M, d(e^s\alpha))$ as the \emph{symplectization part} and $(W,\omega)$ as the \emph{cobordism part} of the completion.

\subsection{Liouville hypersurfaces and convex gluing} \label{Avdek}
Theorem~\ref{thm:main-thm} relies on a result of Avdek \cite{avdek2021liouville}. This section reviews the necessary background for stating Avdek's result in the case of a 3-dimensional contact manifold $(M,\xi)$.\\

Recall that a \emph{Liouville domain} is a pair $\Sdom$ where
\begin{enumerate}
\item $\Sigma$ is a smooth, compact manifold with boundary,
\item $\beta\in\Omega^{1}(\Sigma)$ is such that $d\beta$ is a symplectic form on $\Sigma$, and
\item the unique vector field $Z_{\beta}$ satisfying $d\beta(Z_{\beta},\ast) = \beta$ points out of $\partial \Sigma$ transversely.
\end{enumerate}
The vector field $Z_{\beta}$ on $\Sigma$ described above is called the \emph{Liouville vector field} for $\Sdom$.

\begin{remark}
A Liouville domain is an exact filling of its boundary.
\end{remark}

Let $\Mxi$ be a 3-dimensional contact manifold and let $\Sdom$ be a 2-dimensional Liouville domain. A \emph{Liouville embedding} $i:\Sdom\rightarrow \Mxi$ is an embedding $i:\Sigma\rightarrow M$ such that there exists a contact form $\alpha$ for $\Mxi$ for which $i^{*}\alpha=\beta$. The image of a Liouville embedding will be called a \emph{Liouville hypersurface} and will be denoted by $\Sdom\subset\Mxi$.\\

Every Liouville hypersurface $\Sdom\subset \Mxi$ admits a neighborhood of the form
\begin{equation*}
N(\Sigma)=\Sigma\times[-\epsilon,\epsilon] \quad\text{on which}\quad \alpha = dt + \beta
\end{equation*}
where $t$ is a coordinate on $[-\epsilon,\epsilon]$.  After rounding the edges $\partial \Sigma\times\{\pm\epsilon\}$ of $\partial(\Sigma\times[-\epsilon,\epsilon])$, we obtain a neighborhood $\mathcal{N}(\Sigma)$ of $\Sigma$ for which $\partial \mathcal{N}(\Sigma)$ is a smooth convex surface in $\Mxi$, with contact vector field $t\partial_{t} + Z_{\beta}$ and dividing set $\lbrace 0 \rbrace \times \partial\Sigma$.\\

Fix a 2-dimensional Liouville domain $\Sdom$ and a (possibly disconnected) 3-dimensional contact manifold $\Mxi$.  Let $i_{1}$ and $i_{2}$ be Liouville embeddings of $\Sdom$ into $\Mxi$ whose images, which we will denote by $\Sigma_{1}$ and $\Sigma_{2}$, are disjoint.  Let $\alpha$ be a contact form for $\Mxi$ satisfying $i_1^*\alpha=i_2^*\alpha=\beta$.\\

Consider neighborhoods $\mathcal{N}(\Sigma_{1}),\mathcal{N}(\Sigma_{2})\subset M$ as described above.  Taking coordinates $(x,z)$ on the boundary of each such neighborhood, where $x\in \Sigma$, we may consider the mapping
\begin{equation*}\label{Eq:ConnectSum}
\Upsilon:\partial \mathcal{N}(\Sigma_{1})\rightarrow \partial \mathcal{N}(\Sigma_{2}),\quad \Upsilon(x,z)=(x,-z).
\end{equation*}
The map $\Upsilon$ sends
\begin{enumerate}
\item the positive region of $\partial \mathcal{N}(\Sigma_{2})$ to the negative region of $\partial \mathcal{N}(\Sigma_{1})$,
\item the negative region of $\partial \mathcal{N}(\Sigma_{1})$ to the positive region of $\partial \mathcal{N}(\Sigma_{2})$, and
\item the dividing set of $\partial \mathcal{N}(\Sigma_{1})$ to the dividing set of $\partial \mathcal{N}(\Sigma_{2})$
\end{enumerate}
in such a way that we may perform a \emph{convex gluing}.  In other words, the map $\Upsilon$ naturally determines a contact structure $\#_{(\Sdom,(i_{1},i_{2}))}\xi$ on the manifold
\[
\#_{(\Sigma,(i_{1},i_{2}))} M:= \Bigl( M \setminus \bigl(N(\Sigma_{1})\cup N(\Sigma_{2})\bigr)\Bigr) /\sim
\]
where $p\sim \Upsilon(p)$ for $p\in N(\Sigma_{1})$.  Avdek then proves the following in \cite{avdek2021liouville}:

\begin{theorem}[{\cite[Theorem 1.8]{avdek2021liouville}}] \label{Avdek2}
Let $(M,\xi)$ be a closed, possibly disconnected, contact 3-manifold.  Suppose that there are two Liouville embeddings $i_{1},i_{2}:(\Sigma,\beta) \rightarrow (M,\xi)$ with disjoint images.  Then there is an exact symplectic cobordism $(W,\omega)$ whose negative boundary is $(M,\xi)$ and whose positive boundary is $\#_{(\Sigma,\beta)}\ (M,\xi)$.
\end{theorem}

Avdek's cobordism between $(M,\xi)$ and $\#_{((\Sigma,\beta),(i_1,i_2))}(M,\xi)$ is constructed by attaching what he calls a \emph{symplectic handle} to the compact symplectization $([0,1]\times M,d(e^s\alpha))$ of $(M,\xi)$.  Up to edge rounding, the symplectic handle modeled on the Liouville hypersurface $(\Sigma,\beta)$ has the form
\[
(H_\Sigma,\omega_\beta) = ([-1,1]\times\mathcal{N}(\Sigma),d\theta\wedge dz+d\beta),
\]
where $\theta$ is the coordinate on $[-1,1]$ and $\mathcal{N}(\Sigma)$ is an abstract copy of the neighborhood of $\Sigma$ described above.  This handle admits a vector field $V_\beta$ satisfying $\mathcal{L}_{V_\beta}\omega_\beta=\omega_\beta$ and which points transversely out of $H_\Sigma$ along $[-1,1]\times\partial\mathcal{N}(\Sigma)$ and into $H_\Sigma$ along $\{\pm 1\}\times\mathcal{N}(\Sigma)$.  This allows us to attach $(H_\Sigma,\omega_\beta)$ to the compact symplectization of $(M,\xi)$ along the neighborhoods $\mathcal{N}(\Sigma_1),\mathcal{N}(\Sigma_2)\subset\{1\}\times M$.  For full details, see \cite{avdek2021liouville}.

\begin{remark}
In case $(\Sigma,\beta)=(DT^*S^1,\lambda_{can})$, we call $(H_\Sigma,\omega_\beta)$ a \emph{round symplectic 1-handle}.  Attaching a round symplectic 1-handle to a symplectic cobordism is equivalent to attaching a Weinstein 1-handle, followed by a Weinstein 2-handle which passes over the 1-handle.
\end{remark}

\subsection{The role of slope in splitting along a torus}\label{subsec:tightness-splitting}
The result of splitting a tight contact manifold $(M,\xi)$ along a mixed torus $T$ with some slope $s$ is a contact manifold $(M',\xi')$ which need not be tight.  However, we may use the slopes $s_0=-1$, $s_1=\infty$, and $s_2$ of the neighborhood $T^2\times[0,2]$ hypothesized by Theorem~\ref{thm:main-thm} to determine a finite list of slopes $s$ for which $(M',\xi')$ will be tight.\\

Let us write $M'$ as
\[
M'=S_0\cup_{\psi_0} (M\setminus T)\cup_{\psi_1} S_1,
\]
where each of $S_0$ and $S_1$ is a solid torus, and the maps $\psi_i\colon\partial S_i\to\partial(M\setminus T)$ are chosen so that the image of the meridian of $S_i$ is a curve of slope $s$.  In this decomposition, the basic slice $T^2\times[1,2]$ with which we began abuts $S_0$ at $T_0:=\partial S_0$, while $T^2\times[0,1]$ abuts $S_1$ at $T_1:=\partial S_1$.\\

Now consider the Legendrian knot $L_0\subset S_0$ given by the core of $S_0$.  Along $L_0$, the contact planes will have slope $s$, and as we move towards $\partial S_0$, these planes will make clockwise rotations towards the slope $s_1=\infty$, as the dividing curves of $\partial S_0$ have this slope.  As we then continue pushing through the basic slice $T\times[1,2]$, the contact planes continue their clockwise rotation towards $s_2$.  Altogether, in this portion of $(M',\xi')$, the slope of the contact planes rotates from $s$ to $\infty$ to $s_2$.  If $s>s_2$, then this rotation passes through an angle in excess of $\pi$, meaning that $\xi'$ is overtwisted.  Moreover, if $s=s_2$, then restricting $\xi'$ to the union of $S_0$ with $T^2\times[1,2]$ produces a solid torus whose boundary dividing curves are meridional.  Such a solid torus is necessarily overtwisted, so if $\xi'$ is tight, then $s<s_2$.  Similar reasoning applied to $S_1$ and $T^2\times[0,1]$ shows that $s> s_0=-1$, meaning that $0\leq s\leq s_2-1$, provided $\xi'$ is a tight contact structure.

\begin{remark}
Recall that basic slices can be grouped into \emph{continued fraction blocks}, as per \cite[$\S$ 4.4.5]{honda2000classification}, which are tight structures on $T^2\times[0,m]$ with boundary slopes $s_0=-1$ and $s_m=-1-m$, for some integer $m\geq 1$.  This tight structure decomposes into $m$ basic slices $T^2\times[k-1,k]$, $1\leq k\leq m$ and the tori along which basic slices meet have slopes $s_k=-1-k$.  Notice that if $(M,\xi)$ contains a virtually overtwisted continued fraction block, then at least one of the interior tori $T:=T^2\times\{k\}$, $1\leq k\leq m-1$, is mixed.  By applying the change of coordinates
\[
\left(\begin{matrix}
k+1 & 1\\
-1 & 0
\end{matrix}\right)
\]
we ensure that the slope of $T$ is $\infty$, while adjacent slopes in a basic slice decomposition are $-1$ and $1$.  In the notation used above, we see that $s_0=-1$, $s_1=\infty$, and $s_2=1$, meaning that there is a unique splitting slope which produces a tight $(M',\xi')$.
\end{remark}

Because Theorem~\ref{thm:main-thm} obtains $(M',\xi')$ from $(M,\xi)$ by splitting along a torus and relates fillings of these contact manifolds by round symplectic 1-handle attachment, it is natural to wonder whether the attaching link $L_0\sqcup L_1$ for a round symplectic 1-handle determines a splitting slope for the resulting contact manifold.  That is, if attaching a round symplectic 1-handle to $(W',\omega')$ along $L_0\sqcup L_1$ yields $(W,\omega)$, we would like to obtain $\partial(W',\omega')$ by splitting $\partial(W,\omega)$ along the (not-necessarily-mixed) belt torus of the round 1-handle with some particular slope.  Before describing the answer to this question, let us observe a certain asymmetry in the definition of splitting along a torus: when splitting along $T$ with slope $s$, we normalize our identification of $T$ with $\mathbb{R}^2/\mathbb{Z}^2$ so that its dividing curves have slope $\infty$, and so that the dividing curves on the opposite component of a basic slice have slope $-1$.  We could describe the same splitting via a different integer $s$ by using the basic slice on the \emph{other} side of $T$ to normalize our slopes.\\

With this ambiguity addressed, we describe the process of determining the slope $s$.  First, the Legendrian knots $L_0,L_1$ admit standard neighborhoods $N(L_0)$, $N(L_1)$ whose boundaries have dividing slope $1/\mathrm{tb}(L_0)$ and $1/\mathrm{tb}(L_1)$, respectively, where the meridians have slope 0.  Adjacent to $N(L_0)$ is a basic slice whose opposite boundary component has slope $p/q$, for some integers $p,q\geq 1$ with $np-q=1$, and we claim that splitting $\partial(W,\omega)$ along a torus with slope $p-1$ will yield $\partial(W',\omega')$.  Indeed, the transformation
\[
\left(\begin{matrix}
1 & -n\\
p-1 & n-q
\end{matrix}\right)
\]
normalizes the slope of $\partial N(L_0)$ to be $\infty$, while the slope of the opposite boundary component of the adjacent basic slice becomes $-1$.  Under this normalization, the meridian slope of $N(L_0)$ becomes $p-1$, and it is this meridian slope which is named in the recipe for splitting along a torus.

\section{Proof of Theorem~\ref{thm:main-thm}}
Throughout this section we consider $(M,\xi=\ker\alpha)$, a contact manifold with an exact symplectic filling $(W,\omega)$ and mixed torus $T \subset M$. Let $(\widehat W, \widehat \omega)$ be the completion of $(W,\omega)$ and $J$ an adapted almost complex structure on $\widehat W$.  That is,
\begin{itemize}
	\item on $((0,\infty)_s\times M,d(e^s\alpha))$, $J$ is $s$-invariant and satisfies $J\partial_s=R_\alpha$ and $J\xi=\xi$;
	\item on $W\subset\widehat{W}$, $J$ is $\omega$-compatible.
\end{itemize}
During the proof of Theorem~\ref{thm:main-thm} we will need to specify a contact form $\alpha$ and impose additional conditions on $J$, but regularity will still be ensured by the automatic transversality results of Wendl \cite{wendl2008automatic}.\\

The proof of Theorem~\ref{thm:main-thm} proceeds as follows. First we will construct a 1-parameter family
\[
\mathcal{S} = \{u_t : (\mathbb{R} \times S^1,j) \to (\widehat W, J) | du_t \circ j = J \circ du_t, t \in \RR\}
\]
of finite-energy embedded holomorphic cylinders in $(\widehat W,\widehat\omega)$ such that
\begin{enumerate}
	\item[(C1)] When $t \gg 0$ the images $\Sigma_t$ and $\Sigma_{-t}$ of the curves $u_t$ and $u_{-t}$ are in the symplectization $[0,\infty) \times M$.
	
	\item[(C2)] When $t \gg 0$ the projections of $\Sigma_{\pm t}$ under the map $\pi : [0,\infty) \times M \to M$ are $R_+(\tilde{T})$ and $R_-(\tilde{T})$ respectively, where $\tilde{T}\subset M$ is a convex torus isotopic to $T$ through convex tori.
	
	\item[(C3)] $\mbox{Im}(u_t)\cap \mbox{Im}(u_{t'})=\emptyset$ if $t\not = t'$.
\end{enumerate}
We then show that $S = \cup_{t \in \RR} \Sigma_t$ sweeps out a properly embedded solid torus in $(\widehat W, \widehat \omega)$. Finally, we cut $W$ along the solid torus $S' = W \cap S$ and modify the result to obtain a new exact filling.\\

Our first step is to standardize the contact form and almost complex structure on a neighborhood of $T$.  As with contact forms constructed by other authors (c.f. \cite{wendl2010open}, \cite{vaugon2015reeb}, \cite{colin2011sutures}), our form is designed to make possible the holomorphic curve counts in which we are interested.

\begin{lemma}\label{lemma:convex-torus-nbhd}
Let $T\subset(M,\xi)$ be a convex torus in a contact manifold, and assume that the dividing set $\Gamma_T$ consists of two components, which we denote $e_1$ and $e_2$.  Let $a>K>2$ be arbitrary constants.  There is a sutured neighborhood $(N,\Gamma,V(\Gamma))$ of $T$ in $M$ and contact form $\alpha$ for $\xi$ such that:
\begin{enumerate}
	\item $N$ is homeomorphic to $T\times[-\epsilon,\epsilon]$;
	\item the boundary $\partial N$ is concave sutured with respect to $\alpha$ in the component of $V(\Gamma)$ corresponding to $e_1$ and convex sutured with respect to $\alpha$ in the component of $V(\Gamma)$ corresponding to $e_2$;
	\item the components $e_1,e_2$ of $\Gamma_T$ are nondegenerate, elliptic Reeb orbits of Conley-Zehnder index 1;
	\item the actions of $e_1$ and $e_2$ with respect to $\alpha$ are $K$;
	\item the only other Reeb orbits in $N$ of action smaller than $a$ are a pair of hyperbolic orbits $h_2$, $h_2'$ parallel to $e_2$.
\end{enumerate}
\end{lemma}
\begin{proof}
By the flexibility theorem, modulo a perturbation of the convex surface $T$, it suffices to construct an explicit model subject to the condition that $\Gamma_{T}$ consists of two parallel curves of slope $\infty$.  We will construct this model in three steps:  First, we explicitly define a form $\alpha$ on neighborhoods of $e_1$ and $e_2$, then on neighborhoods of the positive and negative regions of $T$.  The result is a concave sutured neighborhood of $T$, so our third step is to perform a concave-to-convex modification near $e_2$.\\

For our first step, we borrow heavily from the proof of \cite[Lemma 4.10]{colin2011sutures}.  Let us denote by $S^1\times D^2$ a component of the neighborhood $N(\Gamma_{T})$, with coordinates $(\theta,\rho,\phi)$ chosen in such a way that $N(\Gamma_{T})\cap T = \{\phi=0\}\cup\{\phi=1/2\}$.  We will define a contact form on $S^1\times D^2$ by carefully choosing functions $f,g\colon [0,1] \to\mathbb{R}$ and declaring $\alpha=f\,d\theta + g\,d\phi$.  The contact condition requires that the quantity
\[
D := f(\rho)\,g'(\rho) - g(\rho)\,f'(\rho)
\]
be everywhere positive, and the Reeb vector field for $\alpha$ is given by
\[
R_\alpha = \dfrac{1}{D}\left(g'(\rho)\,\partial_\theta - f'(\rho)\,\partial_\phi\right).
\]
We now choose $f,g\colon[0,1]\to\mathbb{R}$ such that, for some $\delta>0$,
\begin{enumerate}[label=(\arabic*)]
	\item for $\rho\in[0,1-3\delta]$, $f(\rho)=K(1-\rho^2)$ and $g(\rho)=2a\,\rho^2$;\label{cond:near-core}
	\item for $\rho\in[1-3\delta,1-2\delta]$, $f'(\rho)<0$ and $0<g'(\rho)$;
	\item for $\rho\in[1-2\delta,1]$, $f'(\rho)<0$ and $g(\rho)=2a$;
	\item for $\rho\in[1-\delta,1]$, $f(\rho)=e^{1-\rho}$ and $g(\rho)=2a$.\label{cond:boundary}
\end{enumerate}
Notice that the conditions on $f$ require $3\delta(2-3\delta)\,K>e^\delta$; our hypothesis that $K>2$ allows us to choose $0<\delta<1/3$ such that this inequality is satisfied.  Notice that along $\rho=0$ we have $R_\alpha = \tfrac{1}{K}\,\partial_\theta$, and thus $\{\rho=0\}$ is a Reeb orbit of action $K$, as desired.  For $0 < \rho\leq 1-3\delta$, $R_\alpha=\tfrac{1}{K}\,\partial_\theta + \tfrac{1}{2a}\,\partial_\phi$, and thus any Reeb orbit in this region has action at least $2a$, this being the amount of time required to traverse the $\phi$-direction once.  For $1-3\delta\leq\rho\leq 1-2\delta$, the coefficient of $R_\alpha$ in the $\partial_\phi$-direction is smaller than $1/g$, and for $1-2\delta\leq \rho\leq 1$ this coefficient is exactly $1/g$.  In either case, this coefficient is on the order of $1/(2a)$, and thus any Reeb orbit has action on the order of at least $2a$. We conclude that the only Reeb orbit of action less than $a$ is $\{\rho=0\}$, which has action $K$.  The fact that $\{\rho=0\}$ is nondegenerate with Conley-Zehnder index 1 follows from condition~\ref{cond:near-core} and the fact that $K/(2a)<1$.\\

\begin{figure}
	\centering
	\begin{tikzpicture}[scale=0.15]
\begin{scope}
\clip (-35,-7) rectangle (35,7);

\draw [black, dashed] (-21,0) circle (6);
\draw [black, dashed] (21,0) circle (6);

\draw [black, thick] (-42,0) -- (42,0);


\draw [black, dashed] (-26.315,2.784) -- (-42,2.784);
\draw [black, dashed] (-15.685,2.784) -- (15.685,2.784);
\draw [black, dashed] (26.315,2.784) -- (42,2.784);

\draw [black, dashed] (-26.315,-2.784) -- (-42,-2.784);
\draw [black, dashed] (-15.685,-2.784) -- (15.685,-2.784);
\draw [black, dashed] (26.315,-2.784) -- (42,-2.784);

\draw [black,->] (-38,2) -- (-38,-2);
\draw [black,->] (-34,2) -- (-34,-2);
\draw [black,->] (-30,2) -- (-30,-2);
\draw [black,->] (-12,-2) -- (-12,2);
\draw [black,->] (-8,-2) -- (-8,2);
\draw [black,->] (-4,-2) -- (-4,2);
\draw [black,->] (0,-2) -- (0,2);
\draw [black,->] (4,-2) -- (4,2);
\draw [black,->] (8,-2) -- (8,2);
\draw [black,->] (12,-2) -- (12,2);
\draw [black,->] (30,2) -- (30,-2);
\draw [black,->] (34,2) -- (34,-2);
\draw [black,->] (38,2) -- (38,-2);

\draw [black,->] (-16,-2) -- (-16,2);
\draw [black,->] (-19,5) -- (-23,5);
\draw [black,->] (-26,2) -- (-26,-2);
\draw [black,->] (-23,-5) -- (-19,-5);
\begin{scope}[rotate around={45:(-21,0)}]
\draw [black,->] (-18,-1) -- (-18,1);
\draw [black,->] (-20,3) -- (-22,3);
\draw [black,->] (-24,1) -- (-24,-1);
\draw [black,->] (-22,-3) -- (-20,-3);
\end{scope}
\draw [black,->] (-20,-0.5) -- (-20,0.5);
\draw [black,->] (-20.5,1) -- (-21.5,1);
\draw [black,->] (-22,0.5) -- (-22,-0.5);
\draw [black,->] (-21.5,-1) -- (-20.5,-1);

\begin{scope}[xshift=42cm,yshift=0cm]
\draw [black,<-] (-16,-2) -- (-16,2);
\draw [black,<-] (-19,5) -- (-23,5);
\draw [black,<-] (-26,2) -- (-26,-2);
\draw [black,<-] (-23,-5) -- (-19,-5);
\begin{scope}[rotate around={45:(-21,0)}]
\draw [black,<-] (-18,-1) -- (-18,1);
\draw [black,<-] (-20,3) -- (-22,3);
\draw [black,<-] (-24,1) -- (-24,-1);
\draw [black,<-] (-22,-3) -- (-20,-3);
\end{scope}
\draw [black,<-] (-20,-0.5) -- (-20,0.5);
\draw [black,<-] (-20.5,1) -- (-21.5,1);
\draw [black,<-] (-22,0.5) -- (-22,-0.5);
\draw [black,<-] (-21.5,-1) -- (-20.5,-1);
\end{scope}

\draw [blue, opacity=0.5] (20.5,5.979) -- (20.5,4.866);
\draw [blue, opacity=0.5] (20.5,4.866) arc (120:420:1);
\draw [blue, opacity=0.5] (21.5,5.979) -- (21.5,4.866);
\draw [blue, opacity=0.5] (21.886,0.464) arc (27.65:152.35:1);
\draw [blue, opacity=0.5] (21.886,0.464) -- (26.315,2.784);
\draw [blue, opacity=0.5] (20.114,0.464) -- (15.685,2.784);
\draw [blue, opacity=0.5] (26.315,2.784) arc (27.65:85.22:6);
\draw [blue, opacity=0.5] (15.685,2.784) arc (152.35:94.78:6);

\begin{scope}[xscale=1,yscale=-1]
\draw [blue, opacity=0.5] (20.5,5.979) -- (20.5,4.866);
\draw [blue, opacity=0.5] (20.5,4.866) arc (120:420:1);
\draw [blue, opacity=0.5] (21.5,5.979) -- (21.5,4.866);
\draw [blue, opacity=0.5] (21.886,0.464) arc (27.65:152.35:1);
\draw [blue, opacity=0.5] (21.886,0.464) -- (26.315,2.784);
\draw [blue, opacity=0.5] (20.114,0.464) -- (15.685,2.784);
\draw [blue, opacity=0.5] (26.315,2.784) arc (27.65:85.22:6);
\draw [blue, opacity=0.5] (15.685,2.784) arc (152.35:94.78:6);
\end{scope}


\filldraw[blue, opacity=0.5] (21,2) circle (0.35);
\filldraw[black] (-21,0) circle (0.35);
\filldraw[black] (21,0) circle (0.35);
\filldraw[blue, opacity=0.5] (21,-2) circle (0.35);
\end{scope}

\node at (-40,0) {\small $T^2\times\{0\}$};
\node at (40,0) {\tiny $\Gamma=\{e_1,e_2\}$};

\end{tikzpicture}
	\caption{The sutured neighborhood constructed in Lemma~\ref{lemma:convex-torus-nbhd}.  The keyhole-shaped regions near $e_2$ result from applying a concave-to-convex modification on a Reeb flowbox, each containing a hyperbolic orbit parallel to $e_2$.}
	\label{fig:orbits-near-t}
\end{figure}

Next, we define a contact form on neighborhoods of the positive and negative regions of $T$.  Consider the regions
\[
R_\pm' := R_\pm - N(\Gamma_{T}).
\]
These naturally produce Weinstein domains $(R_\pm',\beta_\pm)$ (though the orientation of $R_-'$ as a Weinstein domain is opposite the orientation coming from $T$), and thus we may write $\alpha=\pm(dt+\beta_\pm)$ on $R_\pm'\times[-\epsilon,\epsilon]$.  Condition~\ref{cond:boundary} ensures that, along $\partial(S^1\times D^2)$, the Reeb vector field is given by $\tfrac{1}{2a}\partial_\phi$.  As a result, we may glue the neighborhoods $R_\pm'\times[-\epsilon,\epsilon]$ to $N(\Gamma_{T})$ in such a way that the coordinate $t$ is identified with $\phi$, and thus obtain a sutured neighborhood of $T$ whose boundary is concave sutured with respect to $\alpha$ near both $e_1$ and $e_2$.  See Figure~\ref{fig:orbits-near-t}.\\

Our final step is to perform a concave-to-convex modification near $e_2$, as in \cite[Proposition 4.6]{colin2011sutures}.  With coordinates as before, we identify a region $S^1\times [1-\delta,1]_\rho\times[0,\phi_1]\subset S^1\times D^2$ in $N(e_2)$ to which neither neighborhood $R_{\pm}'\times[-\epsilon,\epsilon]$ has been glued; see Figure~\ref{fig:orbits-near-t}.  This region is a Reeb flowbox as considered by Colin-Ghiggini-Honda-Hutchings in \cite[Proposition 4.6]{colin2011sutures}.  The concave-to-convex modification is then carried out by perturbing the coefficient $f=e^{1-\rho}$ of $d\theta$ in $\alpha=f\,d\theta+g\,d\phi$ as a function of $\rho$ and $\phi$, introducing a canceling pair $h,e$ of critical points.  Each of these critical points of $f$ produces a Reeb orbit parallel to $e_2$, and by excavating a neighborhood of the Reeb orbit corresponding to $e$, Colin-Ghiggini-Honda-Hutchings produce a convex sutured boundary whose convex smoothing is isotopic to the convex smoothing of our original concave sutured boundary.  By applying this modification along each of the two boundary sutures corresponding to $e_2$, we obtain the desired neighborhood $N$.  Note that the Reeb orbit corresponding to the critical point $h$ of $f$ will have action equal to $f(h)$, and thus will be on the order of $e^{1-\rho}\approx 1<K$.
\end{proof}

When a contact form $\alpha$ is understood, we will denote the action of a Reeb orbit $\gamma$ by $\mathcal{A}_\alpha(\gamma)$.  Recall that our mixed torus $T$ admits a pair of bypass (half-)disks $D_\pm$.  We will now use Lemma~\ref{lemma:convex-torus-nbhd}, along with more of the analysis carried out in \cite[Section 4]{colin2011sutures}, to construct a contact form on a neighborhood of $T$ which includes $D_\pm$.

\begin{lemma} \label{lemma:bypasschange}
Let $T \subset (M,\xi)$ be a mixed torus with dividing set $\Gamma$.  There exists a neighborhood $N(T)\subset M$ of $T$, along with a contact form $\alpha$ for $\xi$ on $N(T)$, such that:
\begin{enumerate}
	\item $N(T)$ is diffeomorphic to $T^2\times[-1,1]$, with $T$ identified with $T^2\times\{0\}$;
	\item each of the three tori $T^2\times\{-1,0,1\}$ has dividing set consisting of two elliptic Reeb orbits of Conley-Zehnder index $1$, with those of $T^2\times\{0\}$ called $e_1,e_2$, those of $T^2\times\{\pm 1\}$ called $e^{\pm}_4,e^{\pm}_5$;
	\item $N(T)$ is decomposed into four manifold-with-corners regions $N^\pm_1,N^\pm_2$ as labeled in Figure~\ref{fig:bypass}, such that:
	\begin{itemize}
		\item these regions have pairwise disjoint interiors;
		\item $N_1^+\cap N_1^- =T$;
		\item $\overline{\Sigma}^\pm:=N_1^\pm\cap N_2^\pm$ intersects $T^2\times\{\pm 1\}$ along $e^\pm_4$ and intersects $T$ along $e_1$;
		\item $\partial N_2^\pm \setminus \overline{\Sigma}^\pm = (T^2\times\{\pm 1\})\setminus \{e_4^\pm\}$;
		\item $\Sigma^{\pm}:=\overline{\Sigma}^\pm\setminus\{e_1,e^\pm_4\}$ is a noncompact convex surface with dividing set $\{e^\pm_3\}$, where $e^\pm_3$ is an additional elliptic Reeb orbit of Conley-Zehnder index 1;
	\end{itemize}
	\item the Reeb vector field $R_\alpha$ is positively transverse to $R_+$ and negatively transverse to $R_-$ for each of the convex surfaces $T^2\times\{-1,0,1\}$, $\Sigma^{\pm}$;
	\item there exist hyperbolic orbits $h^+_2,h^-_2,h^+_5$, and $h^-_5$ in $N^+_1,N^-_1,N^+_2$ and $N^-_2$, respectively; they have Conley-Zehnder index $0$ with respect to $T$, and are each parallel to the elliptic orbit with corresponding decorations;
	\item the actions of $e_1,e_2,h_2^\pm$, and $h_5^\pm$ satisfy $\mathcal{A}_{\alpha}(e_1) = \mathcal{A}_\alpha(e_2) > \mathcal{A}_\alpha(h^{\pm}_2)=\mathcal{A}_{\alpha}(h^{\pm}_5)$ and for any Reeb orbit $\gamma$ of $N(T)$ other than these six, $\mathcal{A}_\alpha(\gamma)>\mathcal{A}_\alpha(e_1) + \mathcal{A}_\alpha(h_2^\pm)$.\label{cond:reeb-actions}
\end{enumerate}
\end{lemma}

A schematic picture of the Reeb orbits in $N(T) = N(D_- \cup T \cup D_+)$ is given in Figure~\ref{fig:bypass}.

\begin{figure}
	\centering
	\input{bypass-nbhd.tex}
	\caption{A schematic for a standard neighborhood of a mixed torus $T^2\times \{0\}$.  The surfaces identified in Lemma~\ref{lemma:bypasschange} are labeled on the left, while their dividing sets are labeled (left-to-right) on the right.  A concave-to-convex modification is carried out in the blue keyhole-shaped regions, producing the blue hyperbolic Reeb orbits.}
	\label{fig:bypass}
\end{figure}

\begin{proof}
Because $T$ is a mixed torus, it admits bypass disks $D_\pm$ on opposite sides of $T$ within $(M,\xi)$.  Our neighborhood $N(T)$ is then a neighborhood of $D_-\cup T\cup D_+$, with $T^2\times[0,1]$ a neighborhood of $T\cup D_+$ and $T^2\times[-1,0]$ a neighborhood of $T\cup D_-$.  We will construct the desired contact form $\alpha$ by decomposing each of these two regions as a smoothly canceling 1-/2-handle pair.\\

First, we apply Lemma~\ref{lemma:convex-torus-nbhd} to construct $\alpha$ in neighborhoods of the three tori $T^2\times\{-1,0,1\}$.  In these three neighborhoods we will find all of the Reeb orbits listed in the statement of the present lemma, with the exceptions of $e^{\pm}_3$.  Moreover, the precise control given in Lemma~\ref{lemma:convex-torus-nbhd} over the actions of the Reeb orbits in these neighborhoods allows us to ensure that condition~\ref{cond:reeb-actions} of the present lemma is satisfied, at least in these three neighborhoods.  Namely, we may choose constants $2 < K_0 \ll K_\pm \ll a$ and use the parameters $a,K_0$ when applying Lemma~\ref{lemma:convex-torus-nbhd} to $T^2\times\{0\}$, while using $a,K_\pm$ when applying Lemma~\ref{lemma:convex-torus-nbhd} to $T^2\times\{\pm 1\}$.  The resulting neighborhoods will have the same boundary behavior, since this is determined by the parameter $a$, but the elliptic orbits constituting the dividing set of $T^2\times\{\pm 1\}$ will have much larger action than will those of $T^2\times\{0\}$.  As explained in the proof of Lemma~\ref{lemma:convex-torus-nbhd}, the hyperbolic orbits which result from the concave-to-convex modification will necessarily have action approximately equal to $1<K_0$.\\

Now $N(T^2\times\{0\})$ has two sutured boundary components, each homeomorphic to $T^2$, and each with two sutures.  One of these sutures --- corresponding to $e_2$ --- is convex with respect to $\alpha$, while the other --- corresponding to $e_1$ --- is concave.  We have assumed that the bypass disks $D_\pm$ have the endpoints of their attaching arcs on $e_2$, and thus in the contact handle model for a bypass neighborhood, the contact 1-handle is attached along a pair of points in the suture corresponding to $e_2$.  Note that this symmetry --- the fact that our bypass disks have their endpoints on a common component of the dividing set of $T^2\times\{0\}$ --- is precisely the condition required for a convex torus to be mixed.  Without this condition, the surfaces $\Sigma^{\pm}$ we construct in this lemma would lack a crucial symmetry needed for our holomorphic curve argument later.\\

As depicted in Figure~\ref{fig:1-handle-as-gluing}, contact 1-handle attachment may be carried out as a gluing of a (convex) sutured contact manifold\footnote{An anonymous referee pointed out that attaching a contact 1-handle to a convex suture is more easily done as an \emph{interval-fibered extension}, as defined in \cite[Example 2.10]{colin2011sutures}.  In particular, the relevant Liouville cobordism is a pair-of-pants constructed by attaching a Weinstein 1-handle to the symplectization of the suture.  This simpler attachment works just as well for our purposes, since it introduces no new Reeb orbits, and we thank the referee for this observation.} by identifying gluing data $(P_+,P_-,\phi)$, where $P_+$ and $P_-$ are disks in the positive and negative regions, respectively, of a single component of $\partial N(T^2\times\{0\})$, each adjacent to the suture, and $\phi$ is a rotation by $\pi$.  Following the construction of \cite[Section 4.3]{colin2011sutures}, we may carry out this gluing in such a manner that any Reeb orbit which intersects the surface $P$ resulting from gluing $P_+$ to $P_-$ has arbitrarily large action.  Namely, the Reeb vector field of $\alpha$ points out of $N(T^2\times\{0\})$ along $P_+$ and into $N(T^2\times\{0\})$ along $P_-$, and the construction of \cite[Section 4.3]{colin2011sutures} tells us how to ``stretch in the Reeb direction" before gluing $P_+$ to $P_-$.\\

\begin{figure}
	\centering
	\begin{subfigure}[t]{0.5\textwidth}
		\def\svgwidth{\columnwidth}
	        \centering
\begingroup%
  \makeatletter%
  \providecommand\color[2][]{%
    \errmessage{(Inkscape) Color is used for the text in Inkscape, but the package 'color.sty' is not loaded}%
    \renewcommand\color[2][]{}%
  }%
  \providecommand\transparent[1]{%
    \errmessage{(Inkscape) Transparency is used (non-zero) for the text in Inkscape, but the package 'transparent.sty' is not loaded}%
    \renewcommand\transparent[1]{}%
  }%
  \providecommand\rotatebox[2]{#2}%
  \newcommand*\fsize{\dimexpr\f@size pt\relax}%
  \newcommand*\lineheight[1]{\fontsize{\fsize}{#1\fsize}\selectfont}%
  \ifx\svgwidth\undefined%
    \setlength{\unitlength}{402.66674612bp}%
    \ifx\svgscale\undefined%
      \relax%
    \else%
      \setlength{\unitlength}{\unitlength * \real{\svgscale}}%
    \fi%
  \else%
    \setlength{\unitlength}{\svgwidth}%
  \fi%
  \global\let\svgwidth\undefined%
  \global\let\svgscale\undefined%
  \makeatother%
  \begin{picture}(1,0.51465963)%
    \lineheight{1}%
    \setlength\tabcolsep{0pt}%
    \put(0,0){\includegraphics[width=\unitlength,page=1]{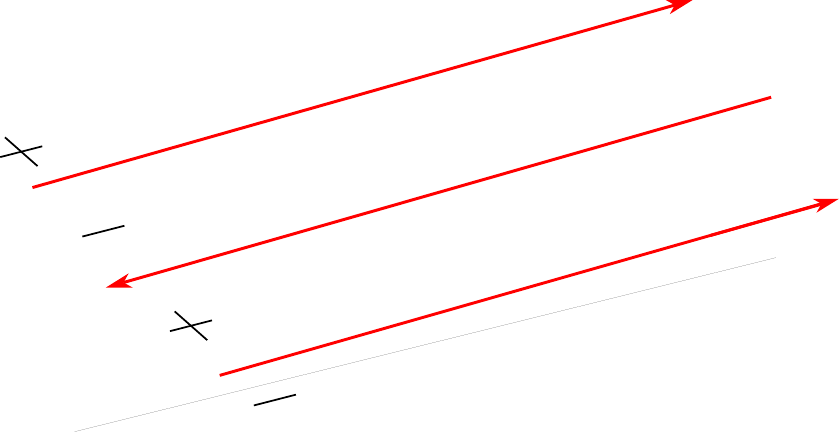}}%
    \put(0.76849832,0.52620567){\color[rgb]{0,0,0}\makebox(0,0)[lt]{\lineheight{1.25}\smash{\begin{tabular}[t]{l}${\color{red}e_2}$\end{tabular}}}}%
    \put(0.87280287,0.40700039){\color[rgb]{0,0,0}\makebox(0,0)[lt]{\lineheight{1.25}\smash{\begin{tabular}[t]{l}${\color{red}e_1}$\end{tabular}}}}%
    \put(0.94358005,0.2877951){\color[rgb]{0,0,0}\makebox(0,0)[lt]{\lineheight{1.25}\smash{\begin{tabular}[t]{l}${\color{red}e_2}$\end{tabular}}}}%
    \put(0,0){\includegraphics[width=\unitlength,page=2]{bypass-data.pdf}}%
  \end{picture}%
\endgroup%

	        \caption{Attaching data for a contact 1-handle.}
    	\end{subfigure}
    ~ 
    \begin{subfigure}[t]{0.5\textwidth}
  \def\svgwidth{\columnwidth}
        \centering
\begingroup%
  \makeatletter%
  \providecommand\color[2][]{%
    \errmessage{(Inkscape) Color is used for the text in Inkscape, but the package 'color.sty' is not loaded}%
    \renewcommand\color[2][]{}%
  }%
  \providecommand\transparent[1]{%
    \errmessage{(Inkscape) Transparency is used (non-zero) for the text in Inkscape, but the package 'transparent.sty' is not loaded}%
    \renewcommand\transparent[1]{}%
  }%
  \providecommand\rotatebox[2]{#2}%
  \newcommand*\fsize{\dimexpr\f@size pt\relax}%
  \newcommand*\lineheight[1]{\fontsize{\fsize}{#1\fsize}\selectfont}%
  \ifx\svgwidth\undefined%
    \setlength{\unitlength}{402.65666811bp}%
    \ifx\svgscale\undefined%
      \relax%
    \else%
      \setlength{\unitlength}{\unitlength * \real{\svgscale}}%
    \fi%
  \else%
    \setlength{\unitlength}{\svgwidth}%
  \fi%
  \global\let\svgwidth\undefined%
  \global\let\svgscale\undefined%
  \makeatother%
  \begin{picture}(1,0.54687882)%
    \lineheight{1}%
    \setlength\tabcolsep{0pt}%
    \put(0,0){\includegraphics[width=\unitlength,page=1]{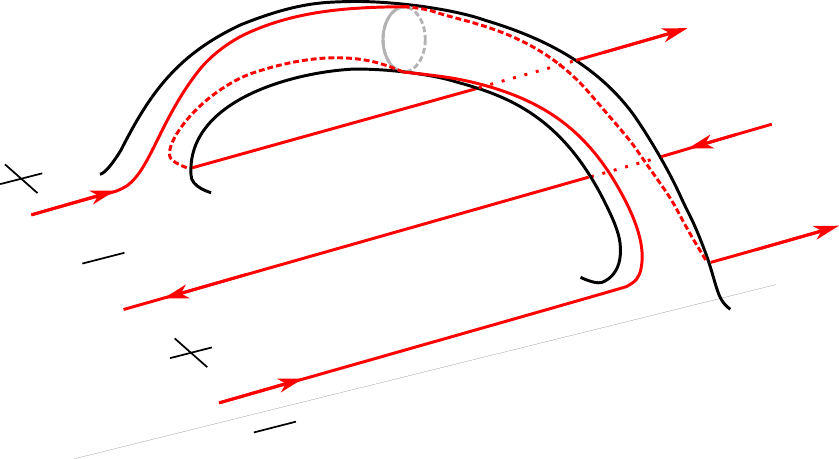}}%
    \put(0.76851754,0.52994416){\color[rgb]{0,0,0}\makebox(0,0)[lt]{\lineheight{1.25}\smash{\begin{tabular}[t]{l}${\color{red}e_3^+}$\end{tabular}}}}%
    \put(0.87282471,0.41446113){\color[rgb]{0,0,0}\makebox(0,0)[lt]{\lineheight{1.25}\smash{\begin{tabular}[t]{l}${\color{red}e_1}$\end{tabular}}}}%
    \put(0.93987845,0.29525297){\color[rgb]{0,0,0}\makebox(0,0)[lt]{\lineheight{1.25}\smash{\begin{tabular}[t]{l}${\color{red}e_4^+}$\end{tabular}}}}%
    \put(0.46116974,0.49081221){\color[rgb]{0,0,0}\makebox(0,0)[lt]{\lineheight{1.25}\smash{\begin{tabular}[t]{l}${\color{gray}P}$\end{tabular}}}}%
  \end{picture}%
\endgroup%

        \caption{The cocore of the 1-handle provides a surface along which we perform a sutured manifold decomposition.}
        \label{fig:1-handle-attached}
    \end{subfigure}
    
    \begin{subfigure}[t]{0.5\textwidth}
  \def\svgwidth{\columnwidth}
        \centering
\begingroup%
  \makeatletter%
  \providecommand\color[2][]{%
    \errmessage{(Inkscape) Color is used for the text in Inkscape, but the package 'color.sty' is not loaded}%
    \renewcommand\color[2][]{}%
  }%
  \providecommand\transparent[1]{%
    \errmessage{(Inkscape) Transparency is used (non-zero) for the text in Inkscape, but the package 'transparent.sty' is not loaded}%
    \renewcommand\transparent[1]{}%
  }%
  \providecommand\rotatebox[2]{#2}%
  \newcommand*\fsize{\dimexpr\f@size pt\relax}%
  \newcommand*\lineheight[1]{\fontsize{\fsize}{#1\fsize}\selectfont}%
  \ifx\svgwidth\undefined%
    \setlength{\unitlength}{402.65666811bp}%
    \ifx\svgscale\undefined%
      \relax%
    \else%
      \setlength{\unitlength}{\unitlength * \real{\svgscale}}%
    \fi%
  \else%
    \setlength{\unitlength}{\svgwidth}%
  \fi%
  \global\let\svgwidth\undefined%
  \global\let\svgscale\undefined%
  \makeatother%
  \begin{picture}(1,0.54405518)%
    \lineheight{1}%
    \setlength\tabcolsep{0pt}%
    \put(0,0){\includegraphics[width=\unitlength,page=1]{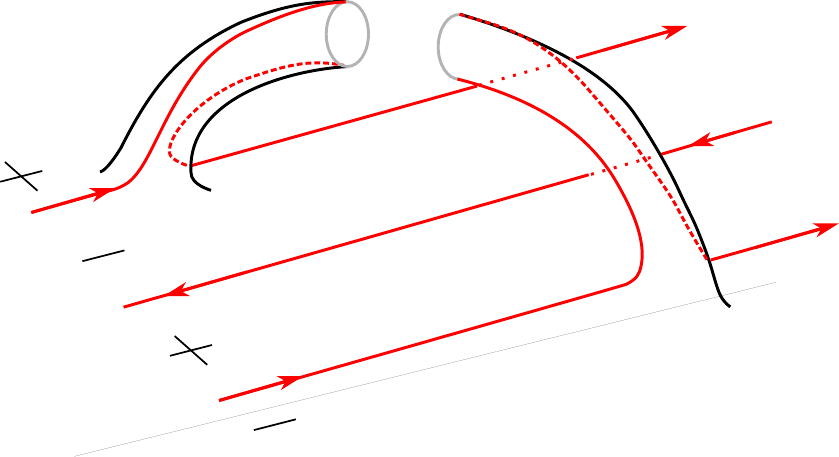}}%
    \put(0.76851755,0.53739462){\color[rgb]{0,0,0}\makebox(0,0)[lt]{\lineheight{1.25}\smash{\begin{tabular}[t]{l}${\color{red}e_2}$\end{tabular}}}}%
    \put(0.87282472,0.41446098){\color[rgb]{0,0,0}\makebox(0,0)[lt]{\lineheight{1.25}\smash{\begin{tabular}[t]{l}${\color{red}e_1}$\end{tabular}}}}%
    \put(0.94360367,0.29897793){\color[rgb]{0,0,0}\makebox(0,0)[lt]{\lineheight{1.25}\smash{\begin{tabular}[t]{l}${\color{red}e_2}$\end{tabular}}}}%
    \put(0.3895548,0.49717496){\color[rgb]{0,0,0}\makebox(0,0)[lt]{\lineheight{1.25}\smash{\begin{tabular}[t]{l}${\color{gray}P_+}$\end{tabular}}}}%
    \put(0,0){\includegraphics[width=\unitlength,page=2]{contact-1-handle-cut.pdf}}%
    \put(0.52200286,0.47623016){\color[rgb]{0,0,0}\makebox(0,0)[lt]{\lineheight{1.25}\smash{\begin{tabular}[t]{l}${\color{gray}P_-}$\end{tabular}}}}%
  \end{picture}%
\endgroup%

        \caption{Performing the decomposition produces a disk in each of the positive and negative regions.}
    \end{subfigure}
    ~ 
    \begin{subfigure}[t]{0.5\textwidth}
  \def\svgwidth{\columnwidth}
        \centering
\begingroup%
  \makeatletter%
  \providecommand\color[2][]{%
    \errmessage{(Inkscape) Color is used for the text in Inkscape, but the package 'color.sty' is not loaded}%
    \renewcommand\color[2][]{}%
  }%
  \providecommand\transparent[1]{%
    \errmessage{(Inkscape) Transparency is used (non-zero) for the text in Inkscape, but the package 'transparent.sty' is not loaded}%
    \renewcommand\transparent[1]{}%
  }%
  \providecommand\rotatebox[2]{#2}%
  \newcommand*\fsize{\dimexpr\f@size pt\relax}%
  \newcommand*\lineheight[1]{\fontsize{\fsize}{#1\fsize}\selectfont}%
  \ifx\svgwidth\undefined%
    \setlength{\unitlength}{402.66674612bp}%
    \ifx\svgscale\undefined%
      \relax%
    \else%
      \setlength{\unitlength}{\unitlength * \real{\svgscale}}%
    \fi%
  \else%
    \setlength{\unitlength}{\svgwidth}%
  \fi%
  \global\let\svgwidth\undefined%
  \global\let\svgscale\undefined%
  \makeatother%
  \begin{picture}(1,0.51465963)%
    \lineheight{1}%
    \setlength\tabcolsep{0pt}%
    \put(0,0){\includegraphics[width=\unitlength,page=1]{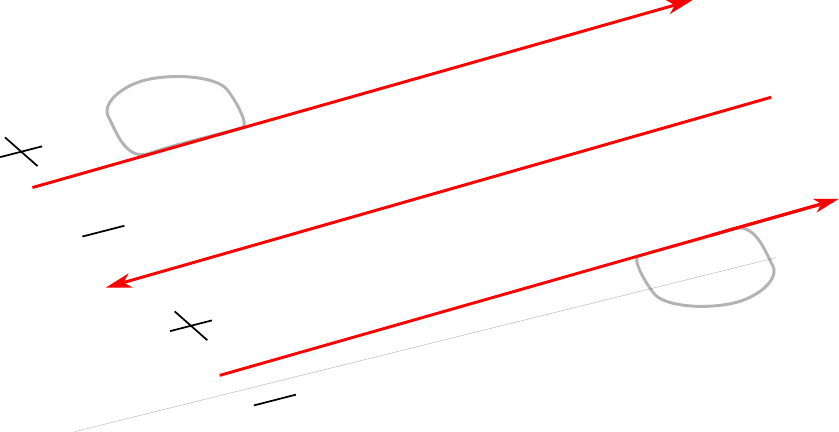}}%
    \put(0.76849831,0.53365611){\color[rgb]{0,0,0}\makebox(0,0)[lt]{\lineheight{1.25}\smash{\begin{tabular}[t]{l}${\color{red}e_2}$\end{tabular}}}}%
    \put(0.87280291,0.41445077){\color[rgb]{0,0,0}\makebox(0,0)[lt]{\lineheight{1.25}\smash{\begin{tabular}[t]{l}${\color{red}e_1}$\end{tabular}}}}%
    \put(0.94358009,0.29524538){\color[rgb]{0,0,0}\makebox(0,0)[lt]{\lineheight{1.25}\smash{\begin{tabular}[t]{l}${\color{red}e_2}$\end{tabular}}}}%
    \put(0.18215358,0.38049394){\color[rgb]{0.6,0.6,0.6}\makebox(0,0)[lt]{\lineheight{1.25}\smash{\begin{tabular}[t]{l}${\color{gray} P_+}$\end{tabular}}}}%
    \put(0.82660699,0.19051046){\color[rgb]{0.6,0.6,0.6}\makebox(0,0)[lt]{\lineheight{1.25}\smash{\begin{tabular}[t]{l}${\color{gray} P_-}$\end{tabular}}}}%
  \end{picture}%
\endgroup%

        \caption{The decomposed surface is identified with the original surface.}
    \end{subfigure}
    \caption{Contact 1-handle attachment may be carried out via the gluing of (convex) sutured contact manifolds, as described in \cite[Section 4.3]{colin2011sutures}.  Each of the surfaces above is assumed to be sutured along its dividing set; while the components of the dividing set are labeled according to parallel Reeb orbits, they are \emph{not} taken to be Reeb themselves.}
    \label{fig:1-handle-as-gluing}
\end{figure}

The gluing just described is carried out on each component of $\partial N(T^2\times\{0\})$, which we think of as attaching a contact 1-handle on either side of $N(T^2\times\{0\})$.  The result is a sutured contact manifold $Y_0$ with two boundary components, each homeomorphic to a surface of genus two, and each with three sutures.  One of the sutures, corresponding to $e_1$, is concave with respect to $\alpha$; the other two sutures, corresponding to $e^{\pm}_3$ and $e^{\pm}_4$, are convex.  We may similarly attach a contact 1-handle to each of $N(T^2\times\{-1\})$ and $N(T^2\times\{1\})$, producing contact manifolds $Y_\pm$ whose boundaries each have a component with one concave suture and two convex sutures.  As depicted in Figure~\ref{fig:orbits-near-t}, $Y_0$, $Y_+$, and $Y_-$ are glued together along their positive and negative boundaries, as well as along the sutures corresponding to $e_1$ and $e^{\pm}_4$.  Note that we are thinking of $Y_0$, $Y_+$, and $Y_-$ as abstractly constructed sutured contact manifolds and then identifying them as subsets of $N(T)$.  In order to realize the contact 1-handles attached to $N(T^2\times\{-1\})$ and $N(T^2\times\{1\})$ as ``upside-down 2-handles", we must identify the positive and negative regions as described; however, the elliptic orbits $e^\pm_4,e^\pm_5$ are not parallel to $e_1,e_2$ under this identification.  Following this gluing, we have nearly completed our construction of $\alpha$ on $N(T)$.  The four remaining sutures correspond to the elliptic Reeb orbits $e^{\pm}_3$ (one on each boundary component of $Y_0$, $Y_\pm$ which is homeomorphic to a surface of genus 2), and thus cannot be directly glued to one another.  However, by repeating the first step of the proof of Lemma~\ref{lemma:convex-torus-nbhd} we may define $\alpha$ on a solid torus with $e^{\pm}_3$ at its core in such a way that this solid torus glues into the boundary component of $Y_-\cup Y_0\cup Y_+$ which corresponds to $e^{\pm}_3$, and such that the action of $e^{\pm}_3$ is as large as that of $e_4^\pm$ or $e_5^\pm$ --- that is, sufficiently large as to satisfy condition~\ref{cond:reeb-actions}.\\

Finally, we describe the surfaces $\Sigma^{\pm}$ and their closures $\overline{\Sigma}^\pm$.  In our construction, each of the Reeb orbits $e_1,e^{\pm}_3$, and $e^{\pm}_4$ has a neighborhood $S^1\times D^2$, with coordinates $(\theta,\rho,\phi)$, such that the Reeb vector field has a positive $\partial_\phi$-component at all points where $\rho\neq 0$.  It follows that $R_\alpha$ is transverse to the interior of any sheet $\{\phi=\phi_0,\rho>0\}$, where $\phi_0$ is constant, and that any such sheet is a convex surface with end $\{\rho=0\}$.  Now consider the surface which results from gluing $Y_0$ to $Y_\pm$, which has three components.  These components meet the neighborhoods of $e_1,e^{\pm}_3$, and $e^{\pm}_4$ just described, and thus may be smoothly extended through these neighborhoods to produce a convex surface $\Sigma^{\pm}$ with dividing set $\{e^{\pm}_3\}$ and cylindrical ends $e_1$ and $e^\pm_4$, as desired.  The closure of $\Sigma^{\pm}$ is $\overline{\Sigma}^{\pm}$ and, along with $T^2\times\{-1,0,1\}$, these decompose $N(T)$ into the desired regions.  See Figure~\ref{fig:bypass}.
\end{proof}

\begin{remark}
As discussed in its proof, Lemma~\ref{lemma:bypasschange} relies crucially on the difference in sign between the basic slices on the two sides of a mixed torus $T$.  Given any bypass half-disk attached to $T$, one could realize $T\times[0,1]$ or $T\times[-1,0]$ as a smoothly cancelling pair of contact handles; but the symmetry between the surfaces $\Sigma^+$ and $\Sigma^-$ in Lemma~\ref{lemma:bypasschange} requires that the bypass half-disks on either side of $T$ have their endpoints on the same component of $\Gamma_T$.  If $T$ sits between basic slices of the same sign, then the endpoints of the corresponding bypass half-disks lie on opposite components of $\Gamma_T$.  In fact, any convex torus $T$ admits on each side a \emph{trivial bypass} (c.f. \cite[Section 1.4]{honda2002gluing}) from which one could construct a cancelling pair of contact handles; in this case, the bypass half-disk has endpoints on distinct components of $\Gamma_T$ and we again fail to obtain the surfaces $\Sigma^{\pm}$.
\end{remark}

Having identified this neighborhood of a mixed torus $T$, we want now to construct a pair of $J$-holomorphic curves in $\mathbb{R}\times M$ which are positively asymptotic to the dividing set $e_1\cup e_2$, and which project to the positive and negative regions $R_\pm$ of $T$.  We will construct these lifts in parts: we first lift $R_\pm$ minus a collar neighborhood, producing $J$-holomorphic curves with bounded coordinates in the symplectization direction, and then construct $J$-holomorphic half-cylinders positively asymptotic to $e_1\cup e_2$ by which we may complete these lifts.\\

The following lemma will allow us to lift the positive and negative regions.

\begin{lemma} \label{lemma:lift-regions}
Let  $(B,\beta = -df \circ j)$ be a 2-dimensional Weinstein domain, where $f: B \to \RR$ is a Morse function such that $\partial B$ is a level set of $f$ and $j$ is an adapted almost complex structure on $(B,\beta)$.  Let $\alpha = dt + \beta$ be a contact form on $[-\epsilon,\epsilon] \times B$, where $t \in [-\epsilon,\epsilon]$. Then there is an adapted almost complex structure on $\RR \times [-\epsilon,\epsilon] \times B$ such that we can lift $B$ to a holomorphic curve by the map $u(\mathbf{x}) = (f(\mathbf{x}),0,\mathbf{x})$.
\end{lemma}

\begin{proof}
The Liouville vector field $X$ for $\beta$ directs the characteristic foliation on $B = \{0\} \times B$ and satisfies $\iota_Xd\beta = \beta$, and thus $\beta(X) = 0$. The Reeb vector field on $[-\epsilon,\epsilon] \times B$ is $\partial_t$. It follows that the contact structure $\ker(\alpha)$ is spanned by $X$ and $jX + g\,\partial_t$, for some function $g: B \to \RR$.  Since $0 = \alpha(jX + g\,\partial_t) = g + \beta(jX) = g + df(X)$ we have that $g = -df(X)$.\\

We want the almost complex structure $J$ to lift $j$ so we specify
\[
J(X) = jX - df(X)\partial_t
\qquad\text{and}\qquad
J(\partial_s) = \partial_t,
\]
and extend this definition by linearity and the property that $J^2=-I$.  In order to verify that $u(\mathbf{x}) = (f(\mathbf{x}),0,\mathbf{x})$ is $J$-holomorphic we check that
\[
(J\circ du)(X) = (du\circ j)(X).
\]
Indeed,
\[
(J\circ du)(X) = J(df(X)\,\partial_s + X) = jX,
\]
and
\[
(du\circ j)(X) = df(jX)\,\partial_s + jX = -\beta(X)\,\partial_s + jX = jX.
\]
This shows that $u$ is $J$-holomorphic.
\end{proof}

We now construct the desired holomorphic curves in the symplectization $\mathbb{R}\times M$.

\begin{lemma}\label{lemma:lifts-with-asymptotic-ends}
There are embedded holomorphic curves $u_{\pm} : \RR \times S^1 \to [0,\infty) \times M$ such that:
\begin{itemize}
	\item $u_{\pm}$ are Fredholm regular and index 2.
	\item $u_{\pm}$ are positively asymptotic to $e_1$ and $e_2$.
	\item The image of $u_{\pm}$ under the projection $\pi	: [0,\infty) \times M \to M$ is $R_\pm(T)$.
\end{itemize} 
\end{lemma}

\begin{proof}
Let us consider the neighborhood of $T$ constructed in Lemma~\ref{lemma:convex-torus-nbhd} and extended by Lemma~\ref{lemma:bypasschange} to $N(T)$.  As in Lemma~\ref{lemma:convex-torus-nbhd}, we consider this neighborhood to be split between the two-component neighborhood $N(\Gamma_T)$ of the dividing set and a Reeb-thickening $N':=(T^2\times[-\epsilon,\epsilon])-N(\Gamma_{T})$ of $R'_\pm\times\{0\}$, the positive and negative regions of $T$, minus collar neighborhoods.  However, we choose $N(\Gamma_T)$ to be strictly smaller than that neighborhood used in the proof of Lemma~\ref{lemma:convex-torus-nbhd}, so that the neighborhood used here does not meet the region where the concave-to-convex modification was carried out in Lemma~\ref{lemma:convex-torus-nbhd}.\\

Because $R_+'\times\{0\}$ and $R_-'\times\{0\}$ are Weinstein domains, Lemma~\ref{lemma:lift-regions} allows us to lift these regions to holomorphic curves in the symplectization which have constant symplectization coordinate at the boundary.  Let us denote this constant coordinate by $a_0$.  We will now construct holomorphic half-cylinders in the symplectization of $N(\Gamma_{T})$ which are asymptotic to the Reeb orbits $e_1,e_2$, and which we may glue to the lifts coming from Lemma~\ref{lemma:lift-regions}.  This argument closely follows~\cite{wendl2010open}. \\

Let us consider a single component of $N(\Gamma_{T})$, identified with $S^1\times D^2$.  Namely, the intersection of $T$ with this component is given by $\{\phi=\phi_0\}\cup\{\phi=\phi_0+1/2\}$, where we use the coordinates $(\theta,\rho,\phi)$ on $S^1\times D^2$.  Compared to the coordinates used in the proof of Lemma~\ref{lemma:convex-torus-nbhd}, the $\rho$-direction has been scaled, to ensure that our $N(\Gamma_T)$ does not meet the region on which the concave-to-convex sutured modification was carried out.  In these coordinates, $\alpha$ is expressed as $\alpha=f(\rho)\,d\theta + g(\rho)\,d\phi$, with $f$ and $g$ as in the proof of Lemma~\ref{lemma:convex-torus-nbhd}.  In particular, $f$ is independent of $\phi$ in the neighborhood $N(\Gamma_T)$ used here.  Now we notice that the vectors
\[
v_1 := \partial_\rho
\quad\text{and}\quad
v_2 := -g(\rho)\,\partial_\theta + f(\rho)\,\partial_\phi
\]
span the contact structure $\ker\alpha$, for $\rho>0$.  Recall that $R_\alpha = \frac{g'}{D}\,\partial_\theta - \frac{f'}{D}\,\partial_\phi$, where $D(\rho)=f(\rho)\,g'(\rho)-f'(\rho)\,g(\rho)$.  We define an almost complex structure $J$ on $\mathbb{R}_a\times(S^1\times D^2)$ by setting
\[
J\,v_1 = \beta(\rho)\,v_2
\quad\text{and}\quad
J\,\partial_a = R_\alpha,
\]
for $\rho>0$, where $\beta(\rho)>0$ is an appropriately chosen smooth function.  In particular, we choose $\beta$ so that this definition of $J$ extends across $\{\rho=0\}$.  (For instance, this can be accomplished by taking $\beta=1/(\rho^2 f)$ near $\{\rho=0\}$.)  We assume that $\beta(\rho)=1$ outside a neighborhood of $\{\rho=0\}$.\\

We may now express the Cauchy-Riemann equations with respect to our chosen coordinates.  In particular, let us write a map $u\colon\mathbb{R}_s\times S^1_t\to[0,\infty)_a\times (S^1\times D^2)$ as
\[
u(s,t) = (a(s,t),\theta(s,t),\rho(s,t),\phi(s,t)).
\]
Then $u$ is $J$-holomorphic if and only if the following equations are satisfied:
\begin{align*}
a_s &= f\theta_t + g\phi_t \quad\quad\quad\quad &
\rho_s &= \frac{1}{\beta D} (f'\theta_t + g'\phi_t) \\
a_t &= -f\theta_s - g\phi_s &
\rho_t &= -\frac{1}{\beta D} (f'\theta_s + g'\phi_s).
\end{align*}
Indeed, using the fact that
\[
J\,\partial_\theta = -f\,\partial_a - \frac{f'}{\beta D}\,\partial_\rho
\quad\text{and}\quad
J\,\partial_\phi = -g\,\partial_a - \frac{g'}{\beta D}\,\partial_\rho,
\]
one may compute
\[
(J\circ du)(\partial_s) = -(f\theta_s+g\phi_s)\partial_a + \frac{1}{D}(g'a_s - g\beta D \rho_s)\partial_\theta - \frac{1}{\beta D}(f' \theta_s + g' \phi_s)\partial_\rho + \frac{1}{D}(f \beta D \rho_s - f' a_s)\partial_\phi
\]
and
\[
(J\circ du)(\partial_t) = -(f\theta_t+g\phi_t)\partial_a + \frac{1}{D}(g'a_t - g\beta D \rho_t)\partial_\theta - \frac{1}{\beta D}(f' \theta_t + g' \phi_t)\partial_\rho + \frac{1}{D}(f \beta D \rho_t - f' a_t)\partial_\phi.
\]
Comparing these vectors to $(du\circ j)(\partial_s)$ and $(du\circ j)(\partial_t)$, respectively, yields a system of eight differential equations.  Indeed, the four equations seen above result from inspecting the $\partial_a$- and $\partial_r$-components of the vectors, while the four equations corresponding to the $\partial_\theta$- and $\partial_\phi$-components comprise an equivalent system.\\

Note that $f,g,D,$ and $\beta$ are now functions of $\rho(s,t)$.  Consider functions $a(s)$ and $\rho(s)$ satisfying the ordinary differential equations
\[
a'(s) = f(\rho),
\quad
\rho'(s) = \dfrac{f'(\rho)}{\beta(\rho)\,D(\rho)}.
\]
We may define $J$-holomorphic curves $u_{\pm\phi_0}\colon\mathbb{R}\times S^1\to [0,\infty)\times(S^1\times D^2)$ via
\[
u_{\pm\phi_0}(s,t) := (a(s),t,\rho(s),\pm\phi_0),
\]
where $\phi_0\in S^1$ is the constant identified above.  In particular, we have holomorphic half-cylinders $u_{\pm\phi_0}$, and the condition $f'(\rho)<0$ ensures that $\rho'(s)<0$ whenever $\rho>0$.  Consequently, $u_{\pm\phi_0}$ is positively asymptotic to $\{\rho=0\}$ as $s$ grows without bound.  Altogether, with $a_0\in[0,\infty)$ the constant identified above, the unique curve $u_{\pm\phi_0}$ with $\rho(0)=1$ and $a(0)=a_0$ is a holomorphic half cylinder which is positively asymptotic to either $e_1$ or $e_2$ as $s\to\infty$ and has $a(s,t)$ and $\phi(s,t)$ contstant near $\{\rho=1\}$.  (In fact $\phi(s,t)$ is constant throughout its domain.)  By construction, the lift of $R'_{\pm}$ taken at the beginning of this argument has $\phi=\pm\phi_0$ along $\{\rho=1\}$, and thus the holomorphic half-cylinder $u_{\pm\phi_0}$ may be smoothly attached to $R'_{\pm}$ to produce the desired curve $u_{\pm}$.\\

We now compute the index of $u_\pm\colon\mathbb{R}\times S^1 \to [0,\infty)\times M$.  As explained in, say, \cite[Equation 3.2]{wendl2010open} and \cite{wendl2008automatic}, this index is given by
\[
\mathrm{ind}(u_\pm) = \chi(\mathbb{R}\times S^1) + 2\,c_1(N_{u_\pm}) + \mu_{\mathrm{CZ}}(e_1) + \mu_{\mathrm{CZ}}(e_2) = 2\,c_1(N_{u_\pm}) + 2,
\]
where $c_1(N_{u_\pm})$ is the relative first Chern number of the normal bundle of $u_\pm$.  We will now show that $N_{u_\pm}$ admits a nonvanishing section which is constant in the asymptotic trivialization, and thus $c_1(N_{u_\pm})=0$.  Indeed, away from its ends, the image of $u_{\pm}$ is transverse to the Reeb vector field $R_\alpha$.  Near the ends of $u_{\pm}$ $R_\alpha$ is parallel to $\partial_\phi$, and our construction of the ends of $u_{\pm}$ ensure that $\partial_\phi$ is transverse to the image of $u_{\pm}$.  So we may construct a smooth vector field equal to $R_\alpha$ away from the ends of $u_{\pm}$ and equal to $\partial_\phi$ in these asymptotic ends; being everywhere transverse to the image of $u_{\pm}$, this vector field is our desired nonvanishing section, and we conclude that $\mathrm{ind}(u_\pm)=2$.\\

Finally, the Fredholm regularity of $u_\pm$ follows from \cite[Proposition 7]{wendl2010open}.
\end{proof}

We would now like to use the curves $u_{\pm}$ constructed in Lemma~\ref{lemma:lifts-with-asymptotic-ends} as endpoints of a family of holomorphic cylinders which sweep out a solid torus in $(\widehat{W},\widehat{\omega})$.  Towards this goal, we will study the moduli space $\mathcal{M}(e_1,e_2)$ of index 2 curves $u: \RR \times S^1 \to \RR \times M$ which are positively asymptotic to $e_1$ and $e_2$ and represent the same homology class as $u_+$ or $u_-$, and the mixed torus condition will preclude undesirable degenerations in this moduli space.\\

In the following lemma, we construct holomorphic curves lifting the positive and negative regions of the surfaces $\Sigma^+$ and $\Sigma^-$ constructed in Lemma~\ref{lemma:bypasschange}.  These are depicted in Figure~\ref{fig:bypass} and should be viewed as ``walls" in the moduli space of curves which we study.

\begin{lemma} \label{lemma:walls}
Let $P$ be a thrice-punctured sphere.  There are embedded holomorphic curves
\[
u^{\pm}_{1,3^+,4^+},u^{\pm}_{1,3^-,4^-}: P \to [0,\infty) \times T^2 \times [-1,1]
\]
such that:
\begin{itemize}
	\item both are Fredholm regular and have index 2;
	\item $u^{\pm}_{1,3^+,4^+}$ is positively asymptotic to $\{e_1,e^+_3,e^+_4\}$ and $u^{\pm}_{1,3^-,4^-}$ is positively asymptotic to $\{e_1,e^-_3,e^-_4\}$;
	\item under the projection $\pi\colon [0,\infty)\times T^2\times[-1,1]\to T^2\times[-1,1]$, the images of $u^{\pm}_{1,3^+,4^+}$ and $u^{\pm}_{1,3^-,4^-}$ are contained in $R_{\pm}(\Sigma^+)$ and $R_{\pm}(\Sigma^-)$, respectively.
\end{itemize}
\end{lemma}

\begin{proof}
We repeat the procedure of Lemma~\ref{lemma:lifts-with-asymptotic-ends}.  Let us denote by $N(\Gamma_{\Sigma^{\pm}})$ a (disconnected) neighborhood of the three elliptic orbits $e_1,e_3^{\pm},e_4^{\pm}$.  Of course this notation is abusive, since neither $e_1$ nor $e_4^{\pm}$ is contained in $\Sigma^\pm$; nonetheless, they are contained in the closure of $\Sigma^\pm$ and serve as its ``dividing set."  We may then let $R'_{\pm}(\Sigma^{\pm})$ be the closure of $R_{\pm}(\Sigma^{\pm}) \setminus N(\Gamma_{\Sigma^\pm})$ and construct a standard neighborhood of $\Sigma^{\pm}$ by taking the union of $N(\Gamma_{\Sigma^\pm})$ with Reeb thickenings of the surfaces $R'_{\pm}(\Sigma^\pm)$.\\
  
As in the proof of Lemma~\ref{lemma:lifts-with-asymptotic-ends}, $R'_{\pm}(\Sigma^{\pm})$ is a Weinstein domain and we may use Lemma~\ref{lemma:lift-regions} to lift these regions to holomorphic curves in the symplectization whose symplectization coordinate is constant at the boundary.  We complete these lifts by attaching holomorphic half-cylinders in the symplectization of $N(\Gamma_{\Sigma^{\pm}})$ which lie over sheets of the form $\{\phi=\phi_0,\rho>0\}$, the construction of which follows as in Lemma~\ref{lemma:lifts-with-asymptotic-ends}.  In particular, neighborhoods of the elliptic orbits $e^\pm_3,e^\pm_4,e^\pm_5$ were all constructed in the same manner as those for $e_1$ and $e_2$ --- modulo occasional concave-to-convex sutured modifications --- and thus the desired holomorphic half-cylinders exist.\\

Once again, the Fredholm regularity of these curves follows from~\cite[Proposition 7]{wendl2010open}, so it remains to compute their indices.  As in the proof of Lemma~\ref{lemma:lifts-with-asymptotic-ends}, we appeal to~\cite[Equation 3.2]{wendl2010open} to compute
\[
\mathrm{ind}(u^{\pm}_{i,j,k}) = \chi(P) + 2\,c_1(N_{u^{\pm}_{i,j,k}}) + \mu_{\mathrm{CZ}}(e_i) + \mu_{\mathrm{CZ}}(e_j) + \mu_{\mathrm{CZ}}(e_k) = -1 + 2\,c_1(N_{u_\pm}) + 3,
\]
where $c_1(N_{u^{\pm}_{i,j,k}})$ is the relative first Chern number of the normal bundle of $u^{\pm}_{i,j,k}$.  Just as in Lemma~\ref{lemma:lifts-with-asymptotic-ends}, we may construct a nonvanishing section of this normal bundle which is constant in the asymptotic trivialization, leading us to conclude that this relative Chern number is trivial and that $\mathrm{ind}(u^{\pm}_{i,j,k}) = 2$.
\end{proof}

We now consider the quotient $\mathcal{M}(e_1,e_2)/\RR$ of $\mathcal{M}(e_1,e_2)$ by $\RR$-translation.  We show that the ``walls" constructed in Lemma~\ref{lemma:walls} preclude the compactness of this quotient, and that our understanding of the Reeb orbits in $T^2\times[-1,1]$ allows us to understand its compactification.  One can learn about the analysis of such moduli spaces in, for instance, \cite{wendl2010lectures}.

\begin{lemma} \label{lemma:compactification}
The compactification $\overline{\mathcal{M}(e_1,e_2)/\RR}$ is the disjoint union of two components $\mathcal{N}_\pm$ containing the equivalence classes of $u_\pm$ up to $\RR$ translation. The boundary $\partial\mathcal{N}_\pm$ consists of 
\begin{itemize}
	\item a two-level building $v^+_{1,\pm} \cup v^+_{0,\pm}$, where $v^+_{1,\pm}$ is the top level consisting of a cylinder positively asymptotic to $e_2$ and negatively asymptotic to $h^+_2$ and $v^+_{0,\pm}$ is the bottom level consisting of a cylinder positively asymptotic to $e_1$ and $h^+_2$ and 
	\item another two-level building $v^-_{1,\pm} \cup v^-_{0,\pm}$ with $h^+_2$ replaced by $h^-_2$.
\end{itemize}
\end{lemma}

\begin{remark}
The superscripts in $v^{\pm}_{0,\pm}$ and $v^{\pm}_{1,\pm}$ indicate whether the given curve lies above (in the symplectization direction) $T^2\times[0,1]$ or $T^2\times[-1,0]$, while the sign in the subscript is determined by whether the curve is associated with $R_+(T)$ or $R_-(T)$.
\end{remark}

\begin{proof}
Throughout the proof, we will think of $\mathcal{M}(e_1,e_2)/\mathbb{R}$ as parametrizing curves $\pi_M\circ u\colon\mathbb{R}\times S^1\to M$, where $\pi_M\colon\mathbb{R}\times M\to M$ is the obvious projection.  Note that each of these "curves" is really an equivalence class of curves in $\mathbb{R}\times M$, modulo $\mathbb{R}$-translation.  Thus we are able to recover some version of positivity of intersections: according to~\cite[Lemma A.3]{wendl2010strongly} (c.f.~\cite{siefring2008relative,siefring2011intersection}), no distinct pair of curves in $\mathcal{M}(e_1,e_2)/\mathbb{R}$ can intersect, provided at least one of the curves has trivial normal Chern number.  In fact, if
\[
\tilde{u}\colon\dot{\Sigma} \to \mathbb{R}\times M
\]
is any $J$-holomorphic curve which intersects neither $u_+$ nor $u_-$, then $\pi_M\circ\tilde{u}$ must be disjoint from any curve $\pi_M\circ u$ in $\mathcal{M}(e_1,e_2)/\mathbb{R}$ to which it is not equal.\\

With this interpretation established, we claim that $\mathcal{M}(e_1,e_2)/\mathbb{R}$ is not compact.  Indeed, suppose $\mathcal{M}(e_1,e_2)/\mathbb{R}$ contains a component diffeomorphic to $S^1$.  Each curve in this component has a cylindrical end which limits to $e_1$, and these ends are pairwise disjoint; it follows that these cylindrical ends sweep out a neighborhood of $e_1$ diffeomorphic to $S^1\times D^2$.  That is, there is a neighborhood $N(e_1)$ of $e_1$, each of whose points lies on the image of some unique curve in $\mathcal{M}(e_1,e_2)/\mathbb{R}$.  But $N(e_1)$ must intersect the neighborhood $N^+_2$ constructed in Lemma~\ref{lemma:bypasschange}, meaning that some curve in $\mathcal{M}(e_1,e_2)/\mathbb{R}$ intersects $\pi_M\circ u^{\pm}_{i,j,k}$ for some admissible $\{i,j,k\}$.  But neither $u_+$ nor $u_-$ intersects any $u^{\pm}_{i,j,k}$, since the former curves project to $T=T^2\times\{0\}$, while the latter project to $\Sigma^\pm$, and we have previously computed that all of these curves have trivial normal Chern number.  So the above discussion of intersections tells us that we have a contradiction.  We conclude that $\mathcal{M}(e_1,e_2)/\mathbb{R}$ is not compact, and the compactification $\overline{\mathcal{M}(e_1,e_2)/\mathbb{R}}$ is therefore a disjoint union of closed intervals.  We now investigate $\partial\overline{\mathcal{M}(e_1,e_2)/\mathbb{R}}$.\\

Recall that the Reeb actions of $e_1$ and $e_2$ satisfy $\mathcal{A}_\alpha(e_1) = \mathcal{A}_\alpha(e_2)$, and that the only Reeb orbits with smaller action\footnote{The orbits $h_2^\pm$ were constructed in Lemma~\ref{lemma:convex-torus-nbhd} to have Reeb action smaller than $\mathcal{A}_\alpha(e_2)$; the analysis carried out in the present lemma tells us that this \emph{must} be the case, since $\mathcal{M}(e_1,e_2)/\mathbb{R}$ is not compact.  Indeed, while we could relax the lower bound $K>2$ by manipulating the role of $\delta$ in the proof of Lemma~\ref{lemma:convex-torus-nbhd}, we cannot bring this lower bound below the action of the hyperbolic orbit created by the concave-to-convex modification.  We thank an anonymous referee for this observation.} are $h_2^+$ and $h_2^-$.  When enumerating the buildings which might appear in $\partial\overline{\mathcal{M}(e_1,e_2)/\mathbb{R}}$, these are the only orbits we need consider.  Moreover, because the curves in $\mathcal{M}(e_1,e_2)/\RR$ are disjoint from the projections $\pi_M(u^{\pm}_{i,j,k})$, their images must be contained in $N_1^+\cup N_1^-$.\\

With these restrictions, we find just four boundary elements.  We see that $\partial\overline{\mathcal{M}(e_1,e_2)/\mathbb{R}}$ can contain a cylinder positively asymptotic to $e_2$ and negatively asymptotic to $h_2^+$ followed by a cylinder positively asymptotic to $e_1$ and $h_2^+$. The same is true for $h_2^+$ replaced by $h_2^-$.  These buildings are the desired boundary elements.  Indeed, these are the only holomorphic buildings contained in $N_1^+$ or $N_1^-$ --- neighborhoods identified in the discussion preceding Lemma~\ref{lemma:walls} --- having at least one level with a curve positively asymptotic to $e_2$.  We conclude that $\overline{\mathcal{M}(e_1,e_2)/\mathbb{R}}$ consists of two closed intervals, as desired.
\end{proof}

In order to cut along $T$, we will push an index 1 family of curves into the filling $(W,\omega)$.  In particular, Lemma~\ref{lemma:lifts-with-asymptotic-ends} provided us with an index 2 family $\mathcal{M}(e_1,e_2)/\mathbb{R}$, and Lemma~\ref{lemma:compactification} then leads us to consider the family $\mathcal{M}_{\widehat{W}}(e_1,h^+_2)$ consisting of holomorphic cylinders in $\widehat{W}$ that limit to $e_1$ and $h^+_2$ at the positive ends and represent the same homology class as $v^+_{0,+}$ or $v^+_{0,-}$.  The proof of the following lemma will make clear that, while our argument requires a basic slice on each side of the mixed torus, the family of holomorphic curves which we construct can be thought of as existing only on one side of the torus.  By replacing $\mathcal{M}_{\widehat{W}}(e_1,h^+_2)$ with $\mathcal{M}_{\widehat{W}}(e_1,h^-_2)$ we can take this family of curves to lie on either side of the torus, but Lemma~\ref{lemma:compactification} used surfaces $\Sigma^\pm$ on either side of $T$ in order to produce the compactification of $\mathcal{M}(e_1,e_2)/\mathbb{R}$.  If the basic slices abutting $T$ were of the same sign, the corresponding bypass half-disks would have their vertices on distinct components of $\Gamma_T$, and the surfaces $\Sigma^\pm$ produced in Lemma~\ref{lemma:bypasschange} would be asymmetric.  The symmetry of these surfaces is a crucial part of the analysis in Lemma~\ref{lemma:compactification}.

\begin{lemma}\label{lemma:1-param-family}
There is a regular 1-parameter family
\[
\mathcal{S} = \{u_t : (\mathbb{R} \times S^1,j) \to (\widehat W, J)\}
\]
of holomorphic cylinders in $(\widehat W,\widehat\omega)$ parametrized by $t \in \mathbb{R}$ satisfying:
\begin{enumerate}
	\item[(C1)] When $t \gg 0$, the images of the curves $u_t$ and $u_{-t}$ are in the symplectization $[0,\infty) \times M$.
	
	\item[(C2)] When $t \gg 0$, the composition $\pi\circ u_{\pm t}$, where $\pi : [0,\infty) \times M \to M$ is the obvious projection, is an embedding with image $R_{\pm}(\tilde{T})$, where $\tilde{T}\subset M$ is a convex torus isotopic to $T$ through convex tori.
	
	\item[(C3)] Each $u_t$ is an embedding, and $\mathrm{im}(u_t)\cap \mathrm{im}(u_{t'})=\emptyset$ if $t\not = t'$.
\end{enumerate}
\end{lemma}

\begin{proof}
Consider the index 1 family $\mathcal{M}_{\widehat{W}}(e_1,h^+_2)$ consisting of holomorphic cylinders in $\widehat{W}$ that limit to $e_1$ and $h^+_2$ at the positive ends and represent the same homology class as $v^+_{0,+}$ or $v^+_{0,-}$ from Lemma~\ref{lemma:compactification}.  (Note that we are considering cylinders in $\widehat{W}$, while Lemma~\ref{lemma:compactification} studied cylinders in the symplectization.)  Our family of holomorphic cylinders will be taken from $\mathcal{M}_{\widehat{W}}(e_1,h^+_2)$. \\

We begin by proving (C3), under the assumption that $u_t$ and $u_{t'}$ are distinct elements of $\mathcal{M}_{\widehat{W}}(e_1,h^+_2)$.  For this we appeal to the intersection theory of punctured holomorphic curves found in \cite{siefring2008relative, siefring2011intersection}; a summary of this theory may be found in \cite[Appendix A.3]{wendl2010strongly}.  Specifically, for any $t, t'$, Siefring provides a homotopy invariant intersection number $i(u_t;u_{t'})$ whose vanishing implies that $u_t$ and $u_{t'}$ have no isolated intersections.  If we take $u_t$ and $u_{t'}$ to be translations of $v^+_{0,+}$ (or of $v^+_{0,-}$) by some $t+c$, where $c$ is a constant, then the images of $\pi_M\circ u_t$ and $\pi_M\circ u_{t'}$ are identical.  So \cite[Lemma A.3]{wendl2010strongly} implies that $i(u_t;u_{t'})=0$, provided that the normal Chern number $c_N(u_t)$ vanishes.  Indeed, from \cite[Equation (A.2)]{wendl2010strongly} we find
\[
2c_N(u_t) = \mathrm{ind}(u_t) - 2 + 2g + \#\Gamma_0 = 1 - 2 + 0 + 1 = 0,
\]
where $g$ is the genus of the domain of $u_t$ and $\Gamma_0$ is the set of even punctures of $u_t$ --- i.e., the set of punctures at which $u_t$ is asymptotic to an even orbit.  The intersection number $i(u_t;u_{t'})$ is invariant under deformations of $u_t$ and $u_{t'}$ through $\mathcal{M}_{\widehat{W}}(e_1,h_2^+)$, provided $t\ne t'$, and thus the images of $u_t$ and $u_{t'}$ have no isolated intersections, for any $t\ne t'$.  Similarly, the vanishing of $i(u_t;u_t)$ implies that the \emph{singularity number} $\mathrm{sing}(u_t)$ must vanish, and because this number is invariant under deformations we conclude that each $u_t$ is an embedding.\\

Next, we claim that $\mathcal{M}_{\widehat{W}}(e_1,h^+_2)$ contains a noncompact component which interpolates between $v^+_{0,+}$ and $v^+_{0,-}$.  The failure of $\mathcal{M}_{\widehat{W}}(e_1,h^+_2)$ to be compact follows as in Lemma~\ref{lemma:compactification}.  Namely, the analysis of the previous paragraph may be applied to show that the elements of $\mathcal{M}_{\widehat{W}}(e_1,h^+_2)$ are disjoint from the ``walls" $u_{\pm},u^\pm_{1,3^+,4^+}$ and $u^{\pm}_{1,3^-,4^-}$ and their $\mathbb{R}$-translations, since the projections of $v^+_{0,+}$ and $v^+_{0,-}$ to $M$ are disjoint from those of these walls, each of the walls has trivial normal Chern number, and intersection numbers are invariant under deformations through the moduli space.  Because the curves in $\mathcal{M}_{\widehat{W}}(e_1,h^+_2)$ are pairwise disjoint, an $S^1$-family of curves in $\mathcal{M}_{\widehat{W}}(e_1,h^+_2)$ would have projections to $M$ which encircle $e_1$, requiring intersections with the walls.  We conclude that $\mathcal{M}_{\widehat{W}}(e_1,h^+_2)$ contains the desired noncompact component.\\

Let us now consider the noncompact ends of $\mathcal{M}_{\widehat{W}}(e_1,h^+_2)$.  First we claim that bubbles are not a concern.  Because $(\widehat{W},J)$ has dimension 4, it is semipositive (c.f. \cite[Section 5.1 \& Theorem 5.2.1]{mcduff1994j}) and thus admits no holomorphic spheres of negative index.  By automatic transversality results in dimension 4 (c.f. \cite[Theorem 2.44]{wendl2018holomorphic}, \cite{gromov1985pseudo}, \cite{hofer1997genericity}), this means that all holomorphic spheres in $(\widehat{W},J)$ are regular, and thus that no bubbles appear in the compactification of our index 1 family.  We therefore consider the holomorphic buildings into which elements of $\mathcal{M}_{\widehat{W}}(e_1,h^+_2)$ might break.  Say $w$ is the topmost element of such a building, with image in $\mathbb{R}\times M$.  The positive end of $w$ consists of one or both of the Reeb orbits $e_1$ and $h_2^+$.  The walls constructed above ensure that $\pi_M\circ w$ lies in $N_1^+$, since this is the case for all curves in $\mathcal{M}_{\widehat{W}}(e_1,h^+_2)$, and thus part~\ref{cond:reeb-actions} of Lemma~\ref{lemma:bypasschange} tells us that the negative end of $w$ is made up of some (possibly empty) collection of the orbits $e_1,e_2$, and $h_2^{+}$, lest the total action at the positive end of $w$ fail to exceed that at the negative end.  In fact, those action bounds tell us that the positive end of $w$ must contain $e_1$, since $N_1^+$ contains no Reeb orbits of action smaller than that of $h_2^+$.\\

We use homology calculations to continue our analysis of $w$.  Assume without loss of generality that the slopes of $\Gamma_{T^2\times\{0\}}$ and $\Gamma_{T^2\times\{1\}}$ are 0 and 1, respectively.  Under the identification $H_1(T^2\times[-1,1])\simeq H_1(T^2)\simeq \ZZ^2$, we can take
\[
[e_1] = (0,-1)
\quad\text{and}\quad
[e_2] = (0,1),
\]
meaning that $[h_2^+]=(0,1)$.  Now if the positive end of $w$ is made up of $e_1$ alone, then the action bounds of Lemma~\ref{lemma:bypasschange} tell us that the negative end is either empty or comprised of some number of copies of $h_2^+$.  But there is no nonnegative integer $n\geq 0$ such that
\[
[e_1] = (0,-1) = n\,(0,1) = n\,[h_2^+],
\]
and thus the positive end of $w$ must be made up of both $e_1$ and $h_2^+$.  With this positive end, the action bounds of Lemma~\ref{lemma:bypasschange} leave four possibilities for the negative end:
\begin{enumerate}
\item $e_1$ is alone at the negative end;
\item $e_2$ is alone at the negative end;
\item the negative end is some number of copies of $h_2^+$;
\item the negative end is empty.
\end{enumerate}
The first three possibilities are incompatible with the fact that $w$ has a nullhomologous positive end, since none of these potential negative ends are nullhomologous.  We conclude that the negative end of $w$ is empty, and thus that $w=v^+_{0,+}$ or $w=v^+_{0,-}$ --- that is, no breaking has occurred.\\

So $\mathcal{M}_{\widehat{W}}(e_1,h^+_2)$ provides an interval of holomorphic curves interpolating between $v^+_{0,-}$ and $v^+_{0,+}$, and serves as the middle portion of our family $\mathcal{S}$.  For $t\gg 0$ we take $u_t$ (respectively, \ $u_{-t}$) to be a translation of $v^+_{0,+}$ (respectively, $v^+_{0,-}$) by some $t+c$, where $c$ is a constant, viewed inside the symplectization part $[0,\infty)\times M$. This implies (C1).\\

Next, we consider the projections $\pi\circ u_t$, with $|t|\gg 0$.  Following Hofer-Wysocki-Zehnder \cite{hofer1995properties}, we let $\Pi_\alpha\colon TM\to\xi$ denote the bundle map corresponding to projection along $R_\alpha$ and study the section $\Pi_\alpha\circ d(\pi\circ u_t)$ of $u_t^*\xi$.  Note that this section is not identically zero, since the ends of $u_+$ are asymptotic to distinct Reeb orbits.  In \cite[Section 5]{hofer1995properties}, Hofer-Wysocki-Zehnder define the winding number $\mathrm{wind}_{\Pi_\alpha}(u_t)$ to be the sum of the local indices of the zeros of $\Pi_\alpha\circ d(\pi\circ u_t)$ with respect to any trivialization; these indices are necessarily non-negative.  From \cite[Theorem 5.8]{hofer1995properties} we learn that $\mathrm{wind}_{\Pi_\alpha}(u_t)=0$, since $u_t$ has index 1, and thus conclude that $\Pi_\alpha\circ d(\pi\circ u_t)$ is nowhere zero.\\

Now because the almost complex structure $J$ restricts to an isomorphism of $\xi$, the section $\Pi_\alpha\circ d(\pi\circ u_t)$ has rank 2 wherever it is nonzero.  Indeed, the fact that $\Pi_\alpha\circ (d\pi\circ u_t)$ is nowhere zero means that the Reeb vector field $R_\alpha$ is everywhere transverse to the image of $\pi\circ u_t$.  The upshot is that the image of $\pi\circ u_t$ is a smooth, convex annulus in $M$.  By choosing the functions $f,g$ in Lemma~\ref{lemma:convex-torus-nbhd} to be sufficiently symmetric about the hyperbolic critical point that is introduced, we obtain an almost complex structure $J$ which is symmetric about the Reeb orbits, and thus conclude that $v^+_{0,-}$ and $v^+_{0,+}$ approach the orbits from opposing angles.  As a result, the union of the images of $\pi\circ v^+_{0,\pm}$ with the orbits $e_1$ and $h_2^+$ yields the smooth torus $\tilde{T}$.  Moreover, because $e_1$ and $h_2^+$ are Reeb orbits, the characteristic foliation of $\tilde{T}$ is transverse to these orbits, and indeed these orbits form a dividing set for the characteristic foliation.  So $\tilde{T}$ is convex and (C2) is established.
\end{proof}

We refer to the open solid torus in $\widehat{W}$ swept out by $\mathcal{S}$ as $S$, and we now begin working to remove $S\cap W$ from $W$.  First, however, we modify $W$ slightly.  For some $R\gg 0$, we consider the partial completion
\[
W_R := W \cup ([0,R]\times M).
\]
We denote by $\pi_R$ the projection map $\pi_R\colon[0,R]\times M \to M$, and assume that $R$ has been chosen sufficiently large as to ensure the existence of curves $u_{\pm t_0}$ whose images are contained in $[0,\infty)\times M$ and which satisfy
\[
\mathrm{im}(\pi_R\circ u_{\pm t_0}) = R'_{\pm}(\tilde{T}).
\]
Here, as before, $u_{\pm t_0}$ is an element of the family $\mathcal{S}$ and $R'_{\pm}(\tilde{T})$ is the result of removing from $R_\pm(\tilde{T})$ a small collar neighborhood of the boundary.  Our choice of $R$ ensures that the portion of $u_{\pm t_0}$ truncated by $\pi_R$ lies in the collar neighborhood of $\partial R_\pm(\tilde{T})$.  Next we consider $\tilde{N}(\Gamma_{\tilde{T}})$, a small (half-)tubular neighborhood of $\{R\}\times\Gamma_{\tilde{T}}$ in $W_R$.  We remove this neighborhood to produce
\[
W'_R := W_R - \tilde{N}(\Gamma_{\tilde{T}}).
\]
We decompose the boundary of $W'_R$ along its corners; in particular, we define \emph{horizontal} and \emph{vertical} boundary components
\[
\partial_hW_R' := \partial W'_R - \partial W_R \simeq (S^1\times D^2) \sqcup (S^1\times D^2)
\quad\text{and}\quad
\partial_vW_R' := \partial W'_R - \partial_hW'_R.
\]
Moreover, we assume that $\tilde{N}(\Gamma_{\tilde{T}})$ has been chosen to satisfy $\{R\}\times R_{\pm}'(\tilde{T}) = (\{R\}\times R_{\pm}(\tilde{T})) - \tilde{N}(\Gamma_{\tilde{T}})$.\\

Though possibly dubious, the terminology for $\partial_hW_R'$ and $\partial_vW_R'$ is motivated by interpreting $\mathcal{S}$ as giving us "half an open book decomposition."  We think of the solid torus $S$ as divided between a neighborhood of $\Gamma_{\tilde{T}}$ --- the "binding" --- and a region which is foliated by $\Sigma_t := \mathrm{im}(u_t)$.  The foliated portion is then analogous (at least in the eyes of the authors) to the vertical boundary of a Lefschetz fibration, while $\partial\tilde{N}(\Gamma_{\tilde{T}})$ corresponds to the horizontal boundary.  See Figure~\ref{fig:foliated}.\\

\begin{figure}
	\centering
	\begin{tikzpicture}[scale=0.5, every node/.style={scale=1.25}]
	\draw (0,0) +(90:18cm and 6cm) arc (90:270:18cm and 6cm);	
	\draw[red,thick] (0,0) +(170:5cm and 6cm) arc (170:210:5cm and 6cm);
	\draw[red,thick] (0,0) +(-10:5cm and 6cm) arc (-10:30:5cm and 6cm);
	\draw[blue] (0,0) +(30:5cm and 6cm) arc (30:170:5cm and 6cm);
	\draw[blue] (0,0) +(210:5cm and 6cm) arc (210:350:5cm and 6cm);
	
	\draw[blue] (-1.9,0.02) +(190:5cm and 6cm) arc (261.5:180:13.1cm and 1cm);
	\draw (-4.627,-0.979)+(97:2.25cm and 2.05cm) arc (97:277.5:2.25cm and 2.05cm);
	\draw[red,thick] (0,0) +(190:5cm and 6cm) arc (270:261.5:13.1cm and 1cm);
	\begin{scope}[xshift=-5.5cm,yshift=-0.25cm]
		\draw[blue,rotate=45] (0,0)+(252.5:6.78cm and 1cm) arc (252.5:195:6.78cm and 1cm);
		\draw[red,thick,rotate=45] (0,0)+(270:6.78cm and 1cm) arc (270:252.5:6.78cm and 1cm);
	\end{scope}
	\begin{scope}[xshift=-4.25cm,yshift=-0.25cm]
		\draw[blue,rotate=-45] (0,0)+(255:7.58cm and 1cm) arc (255:180:7.58cm and 1cm);
		\draw[red,thick,rotate=-45] (0,0)+(270:7.58cm and 1cm) arc (270:255:7.58cm and 1cm);
	\end{scope}

	\draw[blue,dashed] (4.8,0.05) +(10:5cm and 6cm) +(96.25:22.8cm and 1cm) arc (96.25:180:22.8cm and 1cm);
	\draw (4.627,0.979)+(97:2.25cm and 2.05cm) arc (97:277.5:2.25cm and 2.05cm);
	\draw[red,thick,dashed] (0,0) +(10:5cm and 6cm) arc (90:96.25:22.8cm and 1cm);
	
	\begin{scope}[xshift=5.15cm,yshift=0.04cm]
		\draw[blue,dashed,rotate=19] (0,0)+(98.5:16.02cm and 1cm) arc (98.5:175:16.02cm and 1cm);
		\draw[red,dashed,thick,rotate=19] (0,0)+(90:16.02cm and 1cm) arc (90:98.5:16.02cm and 1cm);
	\end{scope}
	\begin{scope}[xshift=4.55cm,yshift=0.1cm]
		\draw[blue,dashed,rotate=-19] (0,0)+(99:15cm and 1cm) arc (99:180:15cm and 1cm);
		\draw[red,thick,dashed,rotate=-19] (0,0)+(90:15cm and 1cm) arc (90:99:15cm and 1cm);
	\end{scope}

	\filldraw[red] (-4.92,-1.04) circle (4pt);
	\filldraw[red] (4.92,1.04) circle (4pt);
	
	\draw[->] (-20,-6) -- (-20,6);
	\draw[->] (-18,-7) -- (5,-9);
	
	\draw[right] (-20,4) node {$M$};
	\draw[above] (3,-8.75) node {$\mathbb{R}$};
	\draw (-15,4.1) node {$S'$};	
	
	\draw[right,red] (-4.92,-1.04) node {$e_1$};
	\draw[right,red] (4.92,1.04) node {$h^+_2$};
	\draw[right,blue] (4.3,4.04) node {$\Sigma_{t_0+1}$};
	\draw[right,blue] (4.92,-2.5) node {$\Sigma_{-(t_0+1)}$};
	\draw[left,blue] (-18,0) node {$\Sigma_0$};
	\draw[below] (0,-6.5) node {$\{a=R\}$};
	\end{tikzpicture}
\caption{We remove a neighborhood of $\Gamma_{\tilde{T}}$ from $W_R$ to produce $W_R'$, and then construct a family of annuli $\Sigma_t$ which foliate a solid torus in $W_R'$.}
\label{fig:foliated}
\end{figure}

We are nearly ready to split our symplectic filling, but first must normalize the Liouville form along the region which will be removed.  To this end, the following lemma reparametrizes the foliation of our solid torus in $W_R'$.

\begin{lemma}\label{lemma:foliated-torus}
There exists an embedding $[-(t_0+1),t_0+1]\times\Sigma\subset W'_R$ such that:
\begin{enumerate}[label=(\arabic*)]
	\item each $\Sigma_t:=\{t\}\times\Sigma$ is an annulus and is a symplectic submanifold of $W'_R$, for $t\in[-(t_0+1),t_0+1]$;
	\item $\Sigma_{\pm(t_0+1)}=\{R\}\times R_{\pm}'(\tilde{T})$ is transverse in $\partial W_R'$ to the Reeb vector field;
	\item for $t\in[-(t_0+1),t_0+1]$,
	\[
	\partial\Sigma_t = (S^1\times\gamma(t)) \sqcup (S^1\times\gamma(t)) \subset (S^1\times D^2) \sqcup (S^1\times D^2) = \partial_hW'_R,
	\]
	where $\gamma$ parametrizes an arc from $(-1,0)$ to $(1,0)$ in $\partial D^2$.	
\end{enumerate}
\end{lemma}

\begin{proof}
Note that we already have a family $\Sigma_t:=\mathrm{im}(u_t)\cap W'_R$, $t\in[-t_0,t_0]$, satisfying the first and last conditions.  It remains to reparametrize $\Sigma_t$, for $|t|$ between $t_0$ and $t_0+1$, so that the ends of our family lie in a single level of the symplectization --- that is, so that these ends lie in $\partial W'_R$.\\

Let us consider $u_t\colon (\mathbb{R}\times S^1,j)\to ([0,\infty)\times M,J)$, for some $|t|\gg 0$.  In the proof of Lemma~\ref{lemma:1-param-family} we saw that $\mathrm{im}(u_t)$ is transverse to the distribution $\mathbb{R}\langle \partial_a,R_\alpha\rangle$ in $[0,\infty)\times M$, so we may identify a neighborhood of $\mathrm{im}(u_t)$ with $[0,\infty)_a\times [-\epsilon,\epsilon]_z\times(\mathbb{R}\times S^1)$ in such a manner that the first two components correspond to $\partial_a$ and $R_\alpha$, respectively, and the curve $u_t$ has the form
\[
u_t(\mathbf{x}) = (f(\mathbf{x}) + t, 0, \mathbf{x}),
\]
where $f$ is a Morse function on the annulus which is compatible with the Weinstein structure.  This matches the form of the curves constructed in Lemma~\ref{lemma:lift-regions}.  Let us define $\tilde{f}(\mathbf{x}) := f(\mathbf{x}) + t_0 + 1 - R$ and observe that the image of
\begin{equation}\label{eq:symplectic-surface}
\mathbf{x} \mapsto (R-(t_0+1)+c\,\tilde{f}(\mathbf{x})+t,0,\mathbf{x})
\end{equation}
in $[0,\infty)\times M$ is a symplectic surface, for any $c\in[0,1]$.  Indeed, the symplectic form on $[0,\infty)\times[-\epsilon,\epsilon]\times(\mathbb{R}\times S^1)$ is given by
\[
\omega = d(e^a(dz+\beta)) = e^a\,da\wedge(dz+\beta) + e^a\,d\beta,
\]
where $\beta=-df\circ j$ is our Liouville form on the annulus.  Along the image of the map defined in~(\ref{eq:symplectic-surface}) we have $dz=0$ and $da=c\,d\tilde{f}=c\,df$, so $\omega$ restricts to this surface as $c\,e^a\,df\wedge\beta + e^a\,d\beta$.  Because $df\wedge\beta\geq 0$ and $d\beta>0$, we conclude that this image is a symplectic surface, as desired.\\

We may now interpolate from the current ends $\Sigma_{\pm t_0}$ of our family to ends with the desired properties.  In particular, we may choose a smooth function $c\colon(0,\infty)\to[0,1]$ such that $c(t)=1$ for $t\leq t_0$ and $c(t)=0$ for $t\geq t_0+1$ and set
\[
\Sigma_t := \mathrm{im}(R-(t_0+1)+c(t)\,\tilde{f}(\mathbf{x})+t,0,\mathbf{x})\cap W'_R,
\]
for $t\in(t_0,t_0+1]$.  Notice that $\Sigma_{t_0+1} = \mathrm{im}(R,0,\mathbf{x})\cap W'_R$, as desired.  Following an analogous interpolation at the other end of our family of annuli, we obtain the desired embedding.
\end{proof}

We now have a solid torus $S':=[-(t_0+1),t_0+1]\times\Sigma$ in $W'_R$, and we will remove a neighborhood $N(S')$ of $S'$ to produce the filling $W'$ promised by Theorem~\ref{thm:main-thm}.  Topologically, $N(S')\simeq S^1\times D^1\times D^2$ is a round 1-handle; it remains to realize $N(S')$ as a \emph{symplectic} round 1-handle.  This is accomplished with the following lemma, the proof of which was outlined to the authors by Ko Honda; we also thank Hyunki Min for correcting a mistake in an earlier version of this lemma.

\begin{lemma}\label{lemma:exact-round-handle}
Suppose that the original filling $(W,\omega)$ is exact, with Liouville form $\beta$.  Then, after adjustments of $S'$ and $W'_R$, there exist:
\begin{itemize}
	\item a neighborhood $N(S')=S'\times[-\delta,\delta]_w\subset W'_R$;
	\item a 1-form $\lambda=\lambda_\Sigma + \lambda_B$ on $N(S')$;
	\item a decomposition of $\partial N(S')$ into $\partial_{\mathrm{in}}N(S')=S'\times\{\pm\delta\}$ and $\partial_{\mathrm{out}}N(S')=(\partial S')\times[-\delta,\delta]$;
\end{itemize}
such that:
\begin{enumerate}
	\item $\{\pm(t_0+1)\}\times\Sigma\times[-\delta,\delta]\subset \partial_v W'_R$ and $[-(t_0+1),t_0+1]\times\partial\Sigma\times[-\delta,\delta] \subset \partial_h W'_R$;
	\item $\lambda_\Sigma$ is a Liouville form for $\Sigma$;
	\item $\lambda_B$ is a 1-form on $B=[-(t_0+1),t_0+1]\times[-\delta,\delta]$;
	\item up to a Liouville homotopy of $\beta$, $\lambda$ agrees with $\beta$ on $N(S')$;
	\item the Liouville vector field $X_\lambda$ points into $N(S')$ along $\partial_{\mathrm{in}}N(S')$ and out of $N(S')$ along $\partial_{\mathrm{out}}N(S')$.
\end{enumerate}
\end{lemma}

\begin{proof}
Because $\{\pm(t_0+1)\}\times\Sigma\times\{0\}$ is transverse in $\partial W_R'$ to the Reeb vector field of $\iota_{\partial W_R'}^*\beta$, we may consider a neighborhood
\[
N(S') = [-(t_0+1),t_0+1]_t\times\Sigma\times[-\epsilon,\epsilon]_w \subset W_R'
\]
of $S'$ with the property that $\partial_w$ is parallel to the Reeb vector field of $\iota_{\partial W_R'}^*\beta$ near $\{t=\pm(t_0+1)\}$.  Namely, our coordinate system is chosen so that
\[
\iota^*_{\{t=T\}}\beta = \lambda_{T,0} + T\,dw,
\]
for any fixed $T$ sufficiently near $\pm(t_0+1)$, where $\lambda_{t,w}$ denotes the restriction of $\beta$ to
\[
\Sigma_{t,w} := \{t\}\times \Sigma\times\{w\} \subset N(S'),
\]
for any $(t,w)\in[-(t_0+1),t_0+1]\times[-\epsilon,\epsilon]$.  In particular, the Liouville forms $\lambda_{T,w}$ are independent of $w\in[-\epsilon,\epsilon]$ for $T$ near $\pm(t_0+1)$.  The construction of $\Sigma_{\pm(t_0+1)}$ in Lemma~\ref{lemma:foliated-torus} is symmetric in $t$, and thus we may denote by $\lambda_\Sigma$ the common 1-form $\lambda_{\pm(t_0+1),w}$.\\

Next, we modify $N(S')$, as well as the identification of each $\Sigma_{t,w}$ with $\Sigma$, to ensure that the restriction of $\lambda_{t,w}$ to $\partial\Sigma_{t,w}$ is independent of $(t,w)$.  To this end, let us consider the $\lambda_{t,w}$-length of $\partial\Sigma_{t,w}$.  The compactness of $[-(t_0+1),t_0+1]_t\times[-\epsilon,\epsilon]_w$ provides a uniform lower bound for the $d\lambda_{t,w}$-area of $\Sigma_{t,w}$, achieved by at least one pair $(t,w)$; by extending the neighborhood of $\Gamma_{\tilde{T}}$ used in the construction of $W'_R$, we may excise a portion of each $\Sigma_{t,w}$ so that the symplectic area of $\Sigma_{t,w}$ --- and thus the $\lambda_{t,w}$-length of $\partial\Sigma_{t,w}$ --- is independent of $(t,w)$.  Having made this modification, let us denote the present identification of $\Sigma_{t,w}$ with $\Sigma$ (that is, the identification given by the coordinates on $N(S')$) by $\psi_{t,w}\colon\Sigma\to\Sigma_{t,w}$.  Because each of the 1-forms $\psi^*_{t,w}\lambda_{t,w}$ assigns the same length to $\partial\Sigma$ as does $\lambda_\Sigma$, a standard application of the Moser argument produces diffeomorphisms $\tilde{\phi}_{t,w}\colon\partial\Sigma\to\partial\Sigma$ such that $\tilde{\phi}^*_{t,w}((\psi^*_{t,w}\lambda_{t,w})\vert_{\partial\Sigma})=\lambda_\Sigma\vert_{\partial\Sigma}$.\\

Specifically, the Moser argument for each $(t,w)$ proceeds as follows: because $\psi^*_{t,w}\lambda_{t,w}$ and $\lambda_\Sigma$ assign the same length to $\partial\Sigma$, there exists a function $f_{t,w}\colon\partial\Sigma\to\mathbb{R}$ so that $df_{t,w}=(\psi^*_{t,w}\lambda_{t,w}-\lambda_\Sigma)\vert_{\partial\Sigma}$.  We then set
\[
\lambda_{s,(t,w)} := (s\,\psi^*_{t,w}\lambda_{t,w} + (1-s)\lambda_\Sigma)\vert_{\partial\Sigma} = (\psi^*_{t,w}\lambda_{t,w})\vert_{\partial\Sigma} + s\,df_{t,w}
\]
and take $X_{s,(t,w)}$ to be the unique $s$-dependent vector field on $\partial\Sigma$ with $\iota_{X_{s,(t,w)}}\lambda_{s,(t,w)}=-f_{t,w}$.  It then follows that
\[
\dfrac{d}{ds}\lambda_{s,(t,w)} + \mathcal{L}_{X_{s,(t,w)}}\lambda_{s,(t,w)} = df_{t,w} + d(\iota_{X_{s,(t,w)}}\lambda_{s,(t,w)}) = d(f_{t,w} + \iota_{X_{s,(t,w)}}\lambda_{s,(t,w)}) = 0.
\]
Thus we see that the flow $\tilde{\phi}_{s,(t,w)}\colon\partial\Sigma\to\partial\Sigma$, $s\in[0,1]$, of $X_{s,(t,w)}$ is well-defined and has $\tilde{\phi}_{s,(t,w)}^*(\lambda_{s,(t,w)})=\lambda_\Sigma\vert_{\partial\Sigma}$.  In particular, we obtain the diffeomorphism $\tilde{\phi}_{t,w}:=\tilde{\phi}_{1,(t,w)}$.\\

These details are included so that we may observe the smooth dependence of $f_{t,w}$, $X_{s,(t,w)}$, and finally $\tilde{\phi}_{t,w}$ on $(t,w)$.  Finally, we extend the diffeomorphisms $\tilde{\phi}_{t,w}\colon\partial\Sigma\to\partial\Sigma$ to $\phi_{t,w}\colon\Sigma\to\Sigma$ (again, smoothly in $(t,w)$) and replace each identification $\psi_{t,w}$ of $\Sigma_{t,w}$ with $\Sigma$ with the identification $\psi_{t,w}\circ\phi_{t,w}$.  We have now ensured that $\lambda_{t,w}$ restricts to $\partial\Sigma_{t,w}$ in a manner which is independent of $(t,w)$.  Notice that this step adjusts $N(S')$, but leaves $\beta$ unchanged.\\

Now that we have normalized $\beta$ on $\partial([-(t_0+1),t_0+1]\times\Sigma)\times[-\epsilon,\epsilon]$, we begin the normalization process on
\[
S' = [-(t_0+1),t_0+1]\times\Sigma\times\{0\},
\]
We write $\beta\vert_{S'}=\lambda_t+f\, dt$ for some smooth function $f\colon S'\to\mathbb{R}$, where $\lambda_t:=\lambda_{t,0}$.  From our construction of $S'$ in Lemma~\ref{lemma:1-param-family} and Lemma~\ref{lemma:foliated-torus} we see that $\partial_t$ is parallel to the Liouville vector field of $\beta$ for $t$ sufficiently near $\pm(t_0+1)$, and thus $f$ vanishes in this portion of $S'$.  Now along $S'$ we have
\[
d\beta\vert_{S'} = d_2\lambda_t + (d_2f-\dot{\lambda}_t)dt,
\]
with $d_2$ denoting the derivative in the $\Sigma$-direction.  Our earlier normalization efforts ensured that, for each $t\in[-(t_0+1),t_0+1]$, the area form $d_2\lambda_t$ on $\Sigma_{t,0}$ agrees with $d\lambda_\Sigma$ near $\partial\Sigma_{t,0}$.  Moreover, $d_2\lambda_{\pm(t_0+1)}$ agrees with $d\lambda_\Sigma$ on all of $\Sigma_{\pm(t_0+1),0}$.  So $d_2\lambda_t=d\lambda_\Sigma$ in a neighborhood of $\partial S'$, and another standard Moser argument therefore constructs for us a $\Sigma$-fiberwise diffeomorphism of $S'$, supported in the interior of $S'$, after which we have $d_2\lambda_t=d\lambda_\Sigma$, for all $t\in[-(t_0+1),t_0+1]$.\\

Notice that the characteristic line field $\ker(d\beta\vert_{S'})$ is transverse to the $\Sigma$-fibers of $S'$.  Specifically, we now have
\[
d\beta\vert_{S'}=d\lambda_\Sigma+(d_2f-\dot{\lambda}_t)dt,
\]
and therefore $\ker(d\beta\vert_{S'})$ is directed by $\partial_t+X$, where $X$ is the ($t$-dependent) vector field on $\Sigma$ which satisfies $\iota_Xd\lambda_\Sigma=\dot{\lambda}_t-d_2f$.  The Lie derivative $\mathcal{L}_Xd\lambda_\Sigma$ --- computed on $\Sigma$ --- then satisfies
\[
\mathcal{L}_Xd\lambda_\Sigma = d(\iota_Xd\lambda_\Sigma) = d_2(\dot{\lambda}_t-d_2f) = \dfrac{\partial}{\partial t}\left(d_2\lambda_t-d_2d_2f\right) = \dfrac{\partial}{\partial t}\left(d\lambda_\Sigma\right) = 0,
\]
and thus the time-one map of the flow of $X$ is a self-diffeomorphism of $S'$ which preserves the identity $d_2\lambda_t=d\lambda_\Sigma$.  Following this re-identification of $S'$ with $[-(t_0+1),t_0+1]\times\Sigma\times\{0\}$, $\ker(d\beta\vert_{S'})$ is directed by $\partial_t$, and so we have
\[
0 = \iota_{\partial_t}\left(d\beta\vert_{S'}\right) = d_2f-\dot{\lambda}_t.
\]
Summarizing, we have now ensured that $\beta$ restricts to $S'$ as $\beta\vert_{S'}=\lambda_t+f\,dt$ in a manner satisfying
\begin{equation}\label{eq:normalization}
d_2f = \dot{\lambda}_t,
\quad
d_2\lambda_t = d\lambda_\Sigma,
\quad\text{and}\quad
\ker(d\beta\vert_{S'}) = \mathbb{R}\langle\partial_t\rangle,
\end{equation}
with $f$ a function on $S'$ which vanishes for sufficiently large $t$.  Our next goal is to modify $\beta$ so that the coefficient of $dt$ is independent of the $\Sigma$-coordinate.\\

From the second equation of (\ref{eq:normalization}) we see that the difference $\lambda_t-\lambda_\Sigma$ is a closed 1-form on $\Sigma$, for each $t$.  Moreover, because $\lambda_t$ agrees with $\lambda_\Sigma$ in a neighborhood of $\partial\Sigma$, this difference is supported away from $\partial\Sigma$.  Now consider an embedded arc $\gamma$ in $\Sigma$ connecting the two components of $\partial\Sigma$ and notice that
\[
\int_\gamma \left(\lambda_{-(t_0+1)}-\lambda_\Sigma\right) = \int_\gamma 0 = 0.
\]
From the first equation of (\ref{eq:normalization}) we see that
\[
\frac{\partial}{\partial t}\left(\int_\gamma \lambda_t-\lambda_\Sigma\right) = \int_\gamma\tfrac{\partial}{\partial t}\left(\lambda_t-\lambda_\Sigma\right) = \int_\gamma d_2f,
\]
and this last integral is equal to the difference in the values of $f$ at the two endpoints of $\gamma$.  Because $d_2f=\dot{\lambda}_t=0$ in a neighborhood of $\partial\Sigma$, the value of $f$ along each component of $\partial\Sigma$ is constant, and the symmetry of our construction of $W_R'$ --- namely, the symmetry between the two components of $\tilde{N}(\Gamma_{\tilde{T}})$ --- allows us to assume that these two constants are equal.  We conclude that $\tfrac{\partial}{\partial t}\left(\int_\gamma \lambda_t-\lambda_\Sigma\right)=0$, and thus that $\int_\gamma (\lambda_t-\lambda_\Sigma)=0$, for all $-(t_0+1)\leq t\leq t_0+1$.  It follows that the compactly supported 1-form $\lambda_t-\lambda_\Sigma$ admits a compactly supported primitive on $\Sigma$, and we denote by $h\colon S'\to\mathbb{R}$ a compactly supported function satisfying $d_2h=\lambda_t-\lambda_\Sigma$, for each $t$.\\

Now extend $h$ to a function on $W'_R$ which is compactly supported on $N(S')$ away from
\[
(\partial S')\times[-\epsilon,\epsilon] = N(S')\cap\partial W'_R \subset N(S')
\]
and consider the 1-form $\beta':=\beta-dh$.  Because $dh$ vanishes along $\partial W'_R$, $\beta'$ is Liouville homotopic to $\beta$.  Writing
\[
dh = d_2h + \tfrac{\partial h}{\partial t}\,dt + \tfrac{\partial h}{\partial w}\,dw,
\]
we see that
\begin{align*}
\beta'\vert_{S'} &= \beta\vert_{S'} - dh\vert_{S'} = (\lambda_t+f\,dt) - (d_2h-\tfrac{\partial h}{\partial t}\,dt)\vert_{S'}\\
	&= \lambda_t + f\,dt -\lambda_t + \lambda_\Sigma - \tfrac{\partial h}{\partial t}\,dt = \lambda_\Sigma + \left(f-\tfrac{\partial h}{\partial t}\right)\,dt.
\end{align*}
Now $d\beta'=d\beta$, so our earlier normalization of $\beta$ ensures that $\ker(d\beta'\vert_{S'})=\mathbb{R}\langle\partial_t\rangle$.  But
\[
d\beta' = d\lambda_\Sigma + d_2\left(f-\tfrac{\partial h}{\partial t}\right)\wedge dt,
\]
so we conclude that $f-\tfrac{\partial h}{\partial t}$ is independent of the $t$-coordinate, as desired, and also that $d\beta'=d\lambda_\Sigma$.  We note that this step represents our first modification of $\beta$.\\

By a change of coordinates on $N(S')$ we may assume that
\[
d\beta' = d\lambda_\Sigma + dt\wedge dw.
\]
Indeed, a standard application of the Moser technique (c.f. \cite[Exercise 3.36]{mcduff1998introduction}) shows that if $\omega_0,\omega_1$ are symplectic forms on some manifold $X$ which pull back under inclusion to the same closed 2-form on a hypersurface $Y\subset X$ (with or without boundary), then there is a symplectomorphism $\phi\colon(N_0(Y),\omega_0)\to(N_1(Y),\omega_1)$, for some neighborhoods $N_0(Y),N_1(Y)\simeq Y\times(-\epsilon,\epsilon)$ of $Y$, which restricts to the identity on $Y$.  Our change of coordinates follows by applying this fact to $\omega_0=d\lambda_\Sigma+dt\wedge dw$ and $d\beta'$.\\

By performing the above steps parametrically with respect to $w$, we may write
\begin{equation}\label{eq:beta-prime}
\beta' = \lambda_\Sigma + f\,dt + g\,dw
\end{equation}
on all of $N(S')$, for some smooth functions $f,g\colon N(S')\to\mathbb{R}$.  That is, we repeat the repeat the normalization steps which preceded the most recent change of coordinates, this time normalizing $\beta'$ on all slices
\[
S'_w = [-(t_0+1),t_0+1]\times\Sigma\times\{w\}
\]
simultaneously.  Each step is seen to depend smoothly on $w$, but we must carry out this parametric normalization separately from the original normalization on $S'$ because of the change of coordinates in the previous paragraph (which itself required the initial normalization on $S'$).  The one step which requires special attention in the parametric case is the extension of $h$ to $W'_R$.  We want $h$ to satisfy $d_2h=\lambda_{t,w}-\lambda_\Sigma$ for all pairs $(t,w)$, but also need $h$ to be compactly supported.  As before, we may take $h$ to be supported away from $(\partial S')\times\{w\}$, for each $w\in[-\epsilon,\epsilon]$.  However, we must damp $h$ out in the $w$-direction, and thus cannot guarantee the identity $d_2h=\lambda_{t,w}-\lambda_\Sigma$ for all $w\in[-\epsilon,\epsilon]$.  Nonetheless, we can obtain this normalization for all $w\in[-\epsilon/2,\epsilon/2]$ and then redefine $N(S')$ to be this smaller neighborhood where (\ref{eq:beta-prime}) holds.  The damping of $h$ with respect to $w$ affects only those portions of $N(S')$ where $h$ does not vanish, and thus has no effect near $\partial W'_R$.  We conclude that $\beta':=\beta-dh$ is, as in the 0-parametric case, Liouville homotopic to $\beta$.\\

Notice that the identity
\[
d\lambda_\Sigma + dt\wedge dw = d\beta' = d\lambda_\Sigma + d_2f\wedge dt + d_2g\wedge dw + \left(\tfrac{\partial g}{\partial t} - \tfrac{\partial f}{\partial w}\right)\,dt\wedge dw
\]
tells us that $d_2f$ and $d_2g$ vanish everywhere, and thus that we may treat $f$ and $g$ as functions of $t$ and $w$ alone.  We also see that the Liouville vector field $X$ for $\beta'$ is given by
\[
X = X_\Sigma + g\,\partial_t - f\,\partial_w,
\]
where $X_\Sigma$ is the Liouville vector field for $(\Sigma,\lambda_\Sigma)$.  Namely, $\pm g>0$ along $t=\pm(t_0+1)$.\\

It remains to modify $\beta'$ (as well as our identification of $N(S')$) so that $X$ points in along $\partial_{\mathrm{in}}N(S')$.  To this end, let us choose a bump function $\phi\colon[-\epsilon,\epsilon]\to[0,1]$ such that
\begin{itemize}
	\item $\phi\vert_{[-\epsilon/4,\epsilon/4]}\equiv 1$;
	\item $\phi$ vanishes outside of $[-\epsilon/2,\epsilon/2]$.
\end{itemize}
We also take an arbitrarily large constant $N\gg 0$; just how large $N$ must be is determined by $f,g$, and $\phi$, as explained by the computations that follow.  With $\phi$ and $N$ fixed, we define homotopies $f_\tau,g_\tau\colon N(S')\to \mathbb{R}$, $0\leq t\leq 1$, by
\begin{align*}
f_\tau &:= (1-\tau)\,f + \tau\,(\phi(4w)\,w + (1-\phi(4w))\,f)\\
g_\tau &:= (1-\tau)\,g + \tau\,(\phi(w)\,N\,t + (1-\phi(w))\,g).
\end{align*}
We use $f_\tau$ and $g_\tau$ to define a Liouville homotopy $\beta_\tau = \lambda_\Sigma + f_\tau\,dt + g_\tau\, dw$ on $N(S')$, extending trivially to the remainder of $W_R'$.  Indeed, we have
\[
d\beta_\tau = d\lambda_\Sigma + \left(\tfrac{\partial g_\tau}{\partial t}-\tfrac{\partial f_\tau}{\partial w}\right)dt\wedge dw,
\]
from which we see that $d\beta_\tau$ is symplectic if and only if $\tfrac{\partial g_\tau}{\partial t}-\tfrac{\partial f_\tau}{\partial w}>0$.  (Recall that $f$ and $g$ are functions of $t$ and $w$ alone.)  But notice that
\[
\tfrac{\partial g_\tau}{\partial t}-\tfrac{\partial f_\tau}{\partial w} = (1-\tau)\,\left(\tfrac{\partial g_0}{\partial t}-\tfrac{\partial f_0}{\partial w}\right) + \tau\,\left(\tfrac{\partial g_1}{\partial t}-\tfrac{\partial f_1}{\partial w}\right).
\]
Certainly $\tfrac{\partial g_0}{\partial t}-\tfrac{\partial f_0}{\partial w}>0$, and we also have
\[
\tfrac{\partial g_1}{\partial t}-\tfrac{\partial f_1}{\partial w} = \phi(w)\left(N-\tfrac{\partial g}{\partial t}\right) + \left(\tfrac{\partial g}{\partial t}-\tfrac{\partial f}{\partial w}\right) + 4\phi'(4w)\,(f-w) - \phi(4w)\left(1-\tfrac{\partial f}{\partial w}\right).
\]
When $|w|>\epsilon/8$ we have $\phi(4w)=\phi'(4w)=0$, and thus
\[
\tfrac{\partial g_1}{\partial t}-\tfrac{\partial f_1}{\partial w} = \phi(w)\left(N-\tfrac{\partial g}{\partial t}\right)+\left(\tfrac{\partial g}{\partial t}-\tfrac{\partial f}{\partial w}\right).
\]
By choosing $N>\tfrac{\partial g}{\partial t}$ we ensure that each summand is positive.  On the other hand, for $|w|\leq\epsilon/8$ we have $\phi(w)=1$ and thus
\[
\tfrac{\partial g_1}{\partial t}-\tfrac{\partial f_1}{\partial w} = N-\tfrac{\partial f}{\partial w} + 4\phi'(4w)\,(f-w) - \phi(4w)\left(1-\tfrac{\partial f}{\partial w}\right).
\]
We ensure the positivity of this expression by choosing $N\gg 0$.  Because both $\tfrac{\partial g_0}{\partial t}-\tfrac{\partial f_0}{\partial w}$ and $\tfrac{\partial g_1}{\partial t}-\tfrac{\partial f_1}{\partial w}$ are positive, the same is true of $\tfrac{\partial g_\tau}{\partial t}-\tfrac{\partial f_\tau}{\partial w}$, and we conclude that $d\beta_\tau$ is symplectic, for all $0\leq \tau\leq 1$.\\

At last, we verify the Liouville dynamics of $\beta_\tau$.  Because $\pm g\vert_{\{t=\pm(t_0+1)\}}>0$, we have $\pm g_\tau\vert_{\{t=\pm(t_0+1)\}}>0$, and thus the Liouville vector field of $\beta_\tau$ points transversely out of $W_R'$ along $\{t=\pm(t_0+1)\}$, for all $0\leq\tau\leq 1$.  When $\tau=1$ the Liouville vector field is given on $\{|w|\leq\epsilon/32\}$ by
\[
X = X_\Sigma + g_1\,\partial_t - f_1\,\partial_w = X_\Sigma + N\,t\,\partial_t - w\,\partial_w.
\]
Namely, by redefining $N(S')$ to be
\[
N(S') = [-(t_0+1),t_0+1] \times\Sigma \times [-\epsilon/32,\epsilon/32]
\]
and taking $\lambda=\iota^*_{N(S')}\beta_1$, we obtain the data desired by the lemma (with $\delta=\epsilon/32$).
\end{proof}

Lemma~\ref{lemma:exact-round-handle} allows us to realize $N(S')$ as a round symplectic 1-handle in case the original filling $(W,\omega)$ is exact.  In fact, this can be accomplished when $(W,\omega)$ is a weak symplectic filling.

\begin{lemma}\label{lemma:weak-round-handle}
Suppose that the original filling $(W,\omega)$ is weak.  Then, after adjustments of $S'$ and $W'_R$, there exist:
\begin{itemize}
	\item a neighborhood $N(S')=S'\times[-\delta,\delta]_w\subset W'_R$;
	\item a 1-form $\lambda=\lambda_\Sigma + \lambda_B$ on $N(S')$;
	\item a decomposition of $\partial N(S')$ into $\partial_{\mathrm{in}}N(S')=S'\times\{\pm\delta\}$ and $\partial_{\mathrm{out}}N(S')=(\partial S')\times[-\delta,\delta]$;
\end{itemize}
such that:
\begin{enumerate}
	\item $\{\pm(t_0+1)\}\times\Sigma\times[-\delta,\delta]\subset \partial_v W'_R$ and $[-(t_0+1),t_0+1]\times\partial\Sigma\times[-\delta,\delta] \subset \partial_h W'_R$;
	\item $\lambda_\Sigma$ is a Liouville form for $\Sigma$;
	\item $\lambda_B$ is a 1-form on $B=[-(t_0+1),t_0+1]\times[-\delta,\delta]$;
	\item up to a symplectic deformation, $d\lambda$ agrees with $\omega$ on $N(S')$;
	\item the Liouville vector field $X_\lambda$ points into $N(S')$ along $\partial_{\mathrm{in}}N(S')$ and out of $N(S')$ along $\partial_{\mathrm{out}}N(S')$.
\end{enumerate}
\end{lemma}

Lemma~\ref{lemma:weak-round-handle} is due to Hyunki Min and follows from the proof of Lemma~\ref{lemma:exact-round-handle}.  Specifically, if $(W,\omega)$ is a weak filling of its boundary, we may let $\beta$ be any 1-form on $N(S')$ with $d\beta=\iota_{N(S')}^*\omega$ and then repeat the proof of Lemma~\ref{lemma:exact-round-handle} to define $\lambda$, giving $N(S')$ the structure of a round symplectic 1-handle.  If our original symplectic filling is strong, however, then we wish to work with a 1-form $\beta$ defined on a neighborhood of $\partial W_R'\cup S'$ such that $d\beta=\iota^*_{N(\partial W_R'\cup S')}\omega$.  While a strong filling $(W,\omega)$ must carry such a primitive in a neighborhood of its boundary, nontrivial relative second homology may prevent the primitive from extending over a neighborhood of $S'$.  It is for this reason that Theorem~\ref{thm:main-thm} holds for exact and weak fillings, but not for strong symplectic fillings.\\

For concrete examples of the failure of Theorem~\ref{thm:main-thm} for strong fillings, see~\cite{min2022strongly}.  There, Min constructs weakly fillable rational homology spheres $(M,\xi)$ which admit mixed tori.  Our Theorem~\ref{thm:main-thm} will produce a round symplectic 1-handle in any filling $(W,\omega)$ of $(M,\xi)$, and removing this round handle from the filling will yield a weak filling of a new contact manifold $(M',\xi')$.  On the other hand, a result of \cite{ohta1999simple} (suggested by the earlier work \cite{eliashberg1991symplectic}) says that any weak symplectic filling of a rational homology sphere may be perturbed to a strong symplectic filling.  We may carry out this perturbation on $(W,\omega)$ but, in the case of Min's examples, not in a manner compatible with the round symplectic 1-handle in $(W,\omega)$.  Namely, Min's symplectic filling $(W,\omega)$ is obtained by attaching a round symplectic 1-handle to a symplectic filling of a rotative contact structure on a torus bundle; these contact structures are known to be weakly-but-not-strongly fillable (\cite{gay2006four,eliashberg1996unique}).\\

We now explain how to obtain $W$ from $W':=W'_R-N(S')$ using Theorem~\ref{Avdek2}.  We let $M'$ denote the contact boundary $\partial W'$.  By construction, $\partial_{\mathrm{in}}N(S')\cap M'\subset M'$ consists of two disjoint copies of $S'=[-(t_0+1),t_0+1]\times\Sigma$, with contact form given by $\lambda_\Sigma \pm \delta\,dt$, so that $\partial_t$ is parallel to the Reeb direction.  We then apply Theorem~\ref{Avdek2} to $M'$, using the Liouville embedding $\{0\}_t\times\Sigma\hookrightarrow S'$ for each of the two copies of $S'$.  As discussed in Section~\ref{Avdek}, the proof of Theorem~\ref{Avdek2} proceeds by attaching a symplectic handle to a collar neighborhood of $(M',\xi')$ in $(W',\omega')$. After attaching this handle we obtain $(W,\omega)$ with convex boundary $\#_{(\Sigma,\beta)}\ (M',\xi')$, as desired.\\

By construction, $(M',\xi')$ is the result of splitting $(M,\xi)$ along the mixed torus $T$ with some integer slope $s$, and our discussion in Section~\ref{subsec:tightness-splitting} ensures that $0\leq s\leq s_2-1$.  This completes the proof of Theorem~\ref{thm:main-thm}.

\section{Proof of Theorem~\ref{thm:twice-stabilized}}
We will now prove Theorem~\ref{thm:twice-stabilized} using Theorem~\ref{thm:main-thm}.  Throughout, $L$ is an oriented Legendrian knot in a closed, cooriented contact 3-manifold $(M,\xi)$ and $(M',\xi')$ is obtained from $(M,\xi)$ by Legendrian surgery on $S_+S_-(L)$.\\

The first step in applying Theorem~\ref{thm:main-thm} is to identify a mixed torus in $(M',\xi')$.  To this end, let us consider a standard neighborhood $N(S_-(L))\subset M$ of $S_-(L)$ and let $V_0$ be the solid torus obtained from $N(S_-(L))$ by Legendrian surgery along $S_+S_-(L)\subset N(S_-(L))$.  We may then let $V_1 = M - N(S_-(L))$, so that $M' = V_0 \cup V_1$.  The torus $T=\partial N(S_-(L))$ is a mixed torus because stabilizing twice with opposite signs is equivalent to attaching two bypasses with opposite signs.\\

Now let $(W,\omega)$ be an exact filling of $(M',\xi')$.  Theorem~\ref{thm:main-thm} then guarantees that we can decompose $W$ into a manifold $W'$ such that
\[
\partial W' = (S_0 \cup_{\psi_0} V_0) \sqcup (S_1 \cup_{\psi_1} V_1),
\]
where each $S_i$ is a solid torus, and $\psi_i\colon\partial S_i\to\partial V_i$ is an as-yet-undetermined map.  For $i=0,1$, let us write $M_i = S_i \cup_{\psi_i} V_i$.  Notice that, by construction, $\partial W'$ is disconnected.\\

We now claim that the maps $\psi_i$ are uniquely determined.  Take an oriented identification of $\partial N(S_-(L))$ with $T^2=\RR^2/\ZZ^2$ such that the meridian of $N(S_-(L))$ has slope 0 and $\Gamma_{\partial N(S_-(L))}$ has slope $\infty$.  The correspondence between stabilizations and basic slices gives us an embedding of $T^2\times[0,2]$ into $(M',\xi')$ such that $T^2\times\{i\}$ is identified with $\partial N(S_+S_-(L))$, $\partial N(S_-(L))$, and $\partial N(L)$ for $i=0,1$, and 2, respectively.  Moreover, our choice of identification of $\partial N(S_-(L))$ with $T^2$ ensures that $s_0=-1$, $s_1=\infty$, and $s_2=1$.  So Theorem~\ref{thm:main-thm} tells us that $\partial W'$ is the result of splitting $(M,\xi)$ along $\partial N(S_-(L))$ with slope 0.  That is, $M_0\simeq (S^3,\xi_{std})$ and $M_1=(M,\xi)$, and $M_0$ has a unique exact filling.\\

From \cite[Theorem 1.2]{mcduff1991symplectic} (also see work of Gromov \cite{gromov1985pseudo} and Eliashberg \cite{eliashberg1990filling}) we learn that $(S^3,\xi_{std})$ is not \emph{symplectically co-fillable}.  That is, there is no connected, symplectic manifold $(X,\omega)$ with disconnected convex boundary, one of whose boundary components is $(S^3,\xi_{std})$.  Since $\partial W'=(S^3,\xi_{std})\sqcup (M,\xi)$, we conclude that $W'$ is the disjoint union of a filling $(W_0,\omega_0)$ of $(M,\xi)$ and $(B^4,\omega_{std})$.\\

From Theorem~\ref{thm:main-thm} we know that the filling $(W,\omega)$ of $(M',\xi')$ is obtained from $(W',\omega')$ by attaching a round symplectic 1-handle.  Our final claim is that this corresponds to attaching a symplectic 2-handle to $(W_0,\omega_0)$ along $S_+S_-(L)$.  Indeed, round symplectic 1-handle attachment amounts to the attachment of a Weinstein 1-handle, followed by a Weinstein 2-handle passing over that 1-handle.  The 1-handle is attached along two copies of $B^3$, one taken from each of $S_0$ and $S_1$ in $W'$.  The effect of this is to "cancel" $B^4$, leaving us with $(W_0,\omega_0)$.  The solid tori $S_0$ and $S_1$ are joined by the 1-handle to form a single solid torus in $W_0$, the core curve of which is $S_+S_-(L)$.  The round 1-handle attachment is completed by attaching a symplectic 2-handle along this curve.  This proves Theorem~\ref{thm:twice-stabilized}. \qed

\bibliographystyle{alpha}
\bibliography{../references}
\end{document}